\documentclass[11pt,twoside]{article}
\usepackage{fancyhdr,color,graphicx,amsmath,amsfonts,amssymb,latexsym,bm,indentfirst,cite,amsthm,enumerate}
\usepackage[colorlinks=true,backref=page]{hyperref}
\usepackage[all]{xy}
\usepackage[a4paper,left=25mm,right=25mm,top=35mm,body={155mm,230mm}]{geometry}
\usepackage{pdfpages}
\usepackage{titletoc}
\usepackage{fancyhdr}
\newtheorem*{claim}{\hspace{2em} Claim}

\newtheorem{Theorem}{Theorem}[section]
\newtheorem{Definition}[Theorem]{ Definition}
\newtheorem{Lemma}[Theorem]{ Lemma}

\newtheorem{Proposition}[Theorem]{ Proposition}

\newtheorem{theoremalph}{Theorem}

\newtheorem*{Remark}{Remark}

\makeatletter
\newcommand{\subsectionruninhead}{\@startsection{subsection}{2}{0mm}{-\baselineskip}{-0mm}{\bf\large}}
\newcommand{\subsubsectionruninhead}{\@startsection{subsubsection}{3}{0mm}{-\baselineskip}{-0mm}{\bf\normalsize}}
\makeatother

\makeatletter
\newsavebox{\@brx}
\newcommand{\llangle}[1][]{\savebox{\@brx}{\(\m@th{#1\langle}\)}
	\mathopen{\copy\@brx\kern-0.5\wd\@brx\usebox{\@brx}}}
\newcommand{\rrangle}[1][]{\savebox{\@brx}{\(\m@th{#1\rangle}\)}
	\mathclose{\copy\@brx\kern-0.5\wd\@brx\usebox{\@brx}}}
\makeatother

\begin{document}
   \newpage
   \title{The Lyapunov Exponents of Hyperbolic Measures for $C^1$ Star Vector Fields on Three-dimensional Manifolds }
   \author{ Yuansheng Lu  \and Wanlou Wu \footnote{Wanlou Wu is the corresponding author and was supported by NSFC 12001245.}}
   \date{}
   \maketitle

\begin{abstract}
   In this paper, we proved that for every $C^1$ star vector fields on three-dimensional manifolds, every ergodic hyperbolic invariant measure which is not supported on singularities can be approximated by periodic measures, and the Lyapunov exponents of the ergodic hyperbolic invariant measure can also be approximated by the Lyapunov exponents of those periodic measures. 
\end{abstract}

\section{Introduction}
   Let $M^d$ be a $d$-dimensional $C^\infty$ compact Riemannian manifold without boundary. Denote by $\mathfrak{X}^r(M^d)(r\geq 1)$ the space of $C^r$ vector fields on $M^d$. Let $\varphi^X_t=(\varphi_t)_{t\in\mathbb{R}}$ be the $C^1$ flow generated by a vector field $X\in\mathfrak{X}^1(M^d)$, for simplicity, denoted by $\varphi_t$. The derivative of the flow with respect to the space variable is called the \emph{\bf tangent flow} and is denoted by $\Phi_t={\rm d}\varphi^X_t$. Fix a smooth Riemannian metric on $M$, we get a scalar product in every tangent space $T_xM$, $x\in M$, which depends on $x$ in a differentiable way. The limit $$\chi(x,v)=\lim_{t\rightarrow\pm\infty}\dfrac{1}{ t}\log\lVert\Phi_t(v)\rVert,~v\in T_xM,~v\neq0,~x\in M,$$ is called a \emph{\bf Lyapunov exponent} for a tangent vector $v\in T_xM$. The Lyapunov exponents of a differential equation are a natural generalization of the eigenvalues of the matrix in the linear part of the equation. Lyapunov exponents describe asymptotic evolution of a tangent map: positive or negative exponents correspond to exponential growth or decay of the norm, respectively, whereas vanishing exponents mean lack of exponential behavior.
   
   By the Oseledec Theorem \cite{OV}, the limit $\chi(x,v)$ exists for all non-zero vectors $v$ based on almost all state points $x\in M$ with respect to any given invariant measure, and it is independent of the points if the measure is ergodic. Precisely, let $\mu$ be an invariant measure of flow $\varphi_t$, by the Oseledec Theorem \cite{OV}, for $\mu$-almost every $x$, there are a positive integer $k(x)$, real numbers $\chi_1(x)<\chi_2(x)<\cdots<\chi_{k(x)}(x)$ and a measurable $\Phi_t$-invariant splitting $$T_xM=E_1(x)\oplus E_2(x)\oplus\cdots\oplus E_k(x)$$ such that $$\lim_{t\to\pm\infty}\dfrac{1}{ t}\log\parallel\Phi_t(v)\parallel=\chi_i(x),~\forall~v\in E_i(x),~v\neq 0,~i=1,2,\cdots,k(x).$$ The numbers $\chi_1(x),\cdots,\chi_{k(x)}(x)$ are called \emph{\bf Lyapunov exponents} at point $x$ of $\Phi_t$ with respect to $\mu$. Denote by  $$d_i=\text{dim}(E_i(x)),\quad i=1,2,\cdots,k(x)$$ the multiplicities of those Lyapunov exponents, and the vector formed by these numbers (counted with multiplicity, endowed with the increasing order) is called \emph{\bf Lyapunov vector} at point $x$ of $\Phi_t$ with respect to $\mu$. Sometimes, regardless of the multiplicity, we also rewrite the Lyapunov exponents of $\mu$ by $\lambda_1(x)\leq\lambda_2(x)\leq\cdots\leq\lambda_d(x)$. It menas that $$\lambda_j(x)=\chi_i(x),\quad\text{ for each } d_1+d_2+\cdots+d_{i-1}<j\leq d_1+d_2+\cdots+d_i.$$ If $\mu$ is an ergodic invariant measure, then $k(x),\chi_1(x),\cdots,\chi_{k(x)}(x),\lambda_1(x),\cdots,\lambda_d(x)$ are constants which are independent of $x$. The abstract Lyapunov exponent is an active topic in the theory of nonuniformly hyperbolic systems known as Pesin theory that recovers some hyperbolic behavior for the points whose Lyapunov exponents are all non-zero. Furthermore, these points have well-defined unstable and stable invariant manifolds. In view of these reasons, an ergodic invariant measure is called \emph{\bf hyperbolic} if its Lyapunov exponents are different from zero except in the direction of flow.    

   For diffeomorphisms, Anosov \cite{A67} proved that non-wandering orbit segment has a true periodic orbit nearby for the uniformly hyperbolic diffeomorphisms. Katok \cite{Ka} showed that hyperbolic periodic points are dense in the closure of the basin of a given hyperbolic measure for non-uniformly hyperbolic diffeomorphisms. Later, Wang and Sun \cite{WW10} enhanced Katok's work showing that the Lyapunov exponents of an ergodic hyperbolic invariant measure for $C^{1+\alpha}(\alpha>0)$ diffeomorphisms are approximated by those of a hyperbolic atomic measure on a periodic orbit (see also \cite{ZS10} for $C^1$ diffeomorphisms when the Oseledec splitting is dominated). For more approximation properties, one can refer to \cite{LLS09}. Since a vector field is a continuous dynamical system with respect to the time variable and a diffeomorphism can be regarded as a time-$1$ map of some suitable vector field, the dynamical behavior of vector field and diffeomorphism is similar in most cases. By considering the sectional Poincar\'{e} maps of the vector field, sometimes one can get some (not all) similar properties between $d$-dimensional vector fields and $(d-1)$-dimensional diffeomorphisms. It is natural to ask the following two questions:    
\begin{enumerate}
   \item [\rm (1)] Are the hyperbolic periodic points dense in the closure of the basin of a given hyperbolic invariant measure for vector fields?
   
   \item [\rm (2)] Can the Lyapunov exponents of an ergodic hyperbolic invariant measure for vector fields are approximated by those of a hyperbolic atomic measure on a periodic orbit?   
\end{enumerate}   
   
   For the question $(1)$, Ma \cite{M22} stated that hyperbolic periodic points are dense in the closure of the basin of a given hyperbolic measure for smooth semiflows on separable Banach spaces. Recently, Li, Liang and Liu \cite{LLL24} prove that every ergodic hyperbolic invariant measure which is not supported on singularities can be approximated by periodic measures. Comparing with the diffeomorphisms, singularities of vector fields bring more difficulties. Flows with singularities have rich and complicated dynamics such as the \emph{Lorenz attractor} \cite{Lor63,Guc76}. At singularities, one can not define the \emph{linear Poincar\'e flow} (see Definition \ref{Def:linearPoincare}). Hence we lose some compact properties. This prevents one to use some techniques of diffeomorphisms to singular vector fields, such as Crovisier's central model and Pujals-Sambarino's distortion arguments. Although, the sectional Poincar\'{e} maps of the vector field can get similar properties as diffeomorphisms with lower one dimensional, some vector field (the famous Lorenz attractor \cite{Lor63,Guc76}) displays different dynamics. In the spirit of the Lorenz attractor, geometric Lorenz attractors \cite{ABS77,Guc76,GW79} were constructed in a theoretical way. Loosely speaking, a geometric Lorenz attractor is a robust attractor and it contains a hyperbolic singularity which is accumulated by hyperbolic periodic orbits in a robust way. The return map of a geometric Lorenz attractor has some discontinuous points. This fact gives extra diffculties when one wants to generalize Ma\~{n}\'{e} \cite{MR82} classical argument and Katok's \cite{Ka} classical result. 
   
   In this paper, we consider $C^1$ star vector fields on three-dimensional compact Riemannian manifold $M$ and provide a positive responses to these questions $(1)$ and $(2)$. Now, we state the main result of the present paper as the following Theorems \ref{ThmA} and \ref{ThmB}.  
   
\begin{theoremalph}\label{ThmA}
   Let $X$ be a $C^1$ star vector field on three-dimensional compact Riemannian manifold $M$, $\mu$ be an ergodic hyperbolic invariant measure which is not supported on singularities. Then, there is a sequence of periodic measures converging to $\mu$ in the weak$^*$ topology.    
\end{theoremalph}   

   For the question $(2)$, no response has been seen at present. Based on Theorem \ref{ThmA}, we will show in the present paper that the Lyapunov exponents of a hyperbolic ergodic measure are approximated by those of a hyperbolic atomic measure on a periodic orbit. Lyapunov exponents of an atomic measure concentrated on a periodic orbit with period $\pi$ are exactly the logarithm of the norms of eigenvalues of $\Phi_\pi$. Since there is always a zero Lyapunov exponent for $\Phi_t$ along the flow direction, we only consider the Lyapunov exponents of the {\rm (}scaled{\rm )} linear Poincar\'e flow (see Definition \ref{Def:linearPoincare}) with respect to the hyperbolic ergodic measure.      
      	
\begin{theoremalph}\label{ThmB}
   Let $X$ be a $C^1$ star vector field on three-dimensional compact Riemannian manifold $M$, $\mu$ be an ergodic hyperbolic invariant measure, which is not supported on singularities, with Lyapunov exponents $\lambda_1\leq\lambda_2$ with respect to the {\rm (}scaled{\rm )} linear Poincar\'e flow. To be precise, for every $\varepsilon>0$, there is a periodic point $p$ with Lyapunov exponents $\lambda_1(p)\leq\lambda_2(p)$ such that $$\lvert\lambda_i-\lambda_i(p)\rvert<\varepsilon,~i=1,2.$$  
\end{theoremalph}

   One of the main difficulties for proving Theorems \ref{ThmA} and \ref{ThmB} is the presence of singularities. Even there is no singularities, we are not able to use the usual Pesin theory as in Lian and Young \cite{Lian} since the vector field is only $C^1$. What we need to do is conquer the difficulties caused by singularities and $C^1$ differentiability. Although, the shadowing lemma for singular flows given by Liao \cite{LST85}provided us with a method to find periodic points, we can not estimate the Lyapunov exponents of flow with respect to periodic measures as diffeomorphisms. Our paper is organized as follows. In Section $2$, we introduce some preliminary facts on vector field. In Section $3$, we introduce Lyapunov metric and shadowing lemma on vector field. Theorems \ref{ThmA} and \ref{ThmB} are proved in Section $4$ and Section $5$, respectively.  
       
\section{Preliminaries}	

\subsection{Basic Contents of Vector Fields}
   Let $M^d$ be a $d$-dimensional $C^\infty$ compact Riemannian manifold without boundary. Denote by $\mathfrak{X}^r(M^d)(r\geq 1)$ the space of all $C^r$ vector fields on $M^d$. Given $X\in\mathfrak{X}^1(M^d)$, a point $x\in M^d$ is a \emph{singularity} if $X(x)=0$. Denote by ${\rm Sing}(X)$ the set of all singularities. A point $x$ is \emph{regular} if $X(x)\neq 0$. Let $\varphi^X_t=(\varphi_t)_{t\in\mathbb{R}}$ be the $C^1$ flow generated by a vector field $X$, for simplicity, denoted by $\varphi_t$. A regular point $p$ is \emph{periodic}, if $\varphi_{t_0}(p)=p$ for some $t_0>0$. A critical point is either a singularity or a periodic point. Denote the \emph{normal bundle} of $X$ by $$\mathcal{N}^X\triangleq\bigcup\limits_{x\in M\setminus{\rm Sing}(X)}\mathcal{N}_x,$$ where $\mathcal{N}_x$ is the orthogonal complement space of the flow direction $X(x)$, i.e., $$\mathcal{N}_x=\{v\in T_{x}M: v\perp X(x)\}.$$
     
   For the flow $\varphi^X_t$ generated by $X$, its derivative with respect to the space variable is called the \emph{tangent flow} and is denoted by $\Phi_t={\rm d}\varphi^X_t$.
     
\begin{Definition}\label{Def:linearPoincare}
   Given $x\in M\setminus{\rm Sing}(X),v\in\mathcal{N}_x$ and $t\in\mathbb{R}$, the \emph{\bf linear Poincar\'e flow} $$\psi_t:\mathcal{N}\rightarrow\mathcal{N}$$ is defined as $$\psi_t(v)\triangleq\Phi_t(v)-\frac{\langle\Phi_t(v),X(\varphi_t(x)) \rangle}{\parallel X(\varphi_t(x))\parallel^2} X(\varphi_t(x)).$$ Namely, $\psi_t(v)$ is the orthogonal projection of $\Phi_t(v)$ on $\mathcal{N}_{\varphi_t(x)}$ along the flow direction $X(\varphi_t(x))$.  
\end{Definition}  
     
   Fix $T>0$, the norm $$\lVert\psi_T\rVert=\sup\left\{\lvert\psi_T(v)\rvert:~v\in\mathcal{N},~\lvert v\rvert=1\right\}$$ is uniformly upper bounded on $\mathcal{N}$, although $M\setminus{\rm Sing}(X)$ may be not compact. Denote by $$m(\psi_T)=\inf\left\{\lvert\psi_T(v)\rvert:~v\in\mathcal{N},~\lvert v\rvert=1\right\}$$ the mininorm of $\psi_T$. Since $$m(\psi_T)=\lVert\psi_T^{-1}\rVert^{-1}=\lVert\psi_{-T}\rVert^{-1},$$ the mininorm $m(\psi_T)$ is uniformly bounded from $0$ on $\mathcal{N}$. Another useful flow is \emph{\bf scaled linear Poincar\'e flow} $\psi_t^*:\mathcal{N}\rightarrow\mathcal{N}$, which is defined as $$\psi^*_t(v)\triangleq\frac{\parallel X(x)\parallel}{\parallel X(\varphi_t(x))\parallel}\psi_t(v)=\frac{\psi_t(v)}{\parallel\Phi_t|_{\langle X(x)\rangle}\parallel},$$ where $x\in M\setminus{\rm Sing}(X),v\in\mathcal{N}_x$ and $\langle X(x)\rangle$ is the $1$-dimensional subspace of $T_xM$ generated by the flow dieection $X(x)$.
   
   The linear Poincar\'e flow $\psi_t$ loses the compactness due to the existence of singularities. To overcome this difficulty, the linear Poincar\'e flow can also be defined in a more general way by Liao \cite{Liao89}. Li, Gan and Wen \cite{LGW05} used the terminology of \textquotedblleft extended linear Poincar\'e flow\textquotedblright. For every point $x\in M$, the sphere fiber at $x$ is defined as $$S_xM=\{v:~v\in T_xM,~\lvert v\rvert=1\}.$$ Then, the sphere bundle $$SM=\bigcup_{x\in M}S_xM$$ is compact. For any $v\in SM$, one can define the \emph{\bf unit tangent flow} $$\Phi_t^I: SM\rightarrow SM$$ as $$\Phi_t^I(v)=\dfrac{\Phi_t(v)}{\lvert \Phi_t(v)\rvert}.$$ Given a compact invariant set $\Lambda$ of the flow $\varphi_t$, denote by $$\widetilde{\Lambda}=\text{Closure}\left(\bigcup_{x\in \Lambda\backslash \text{Sing}(X)}\dfrac{X(x)}{\lvert X(x)\rvert}\right)$$ in $SM$. Thus the essential difference between $\widetilde{\Lambda}$ and $\Lambda$ is on the singularities. We can get more information on $\widetilde{\Lambda}$: it tells us how regular points in $\Lambda$ accumulate singularities. For every $x\in M$, and any two orthogonal vectors $v_1\in S_xM$, $v_2\in T_xM$, one can define $$\Theta_t(v_1,v_2)=\left(\Phi_t(v_1),\Phi_t(v_2)-\frac{\langle\Phi_t(v_1),\Phi_t(v_2)\rangle}{\parallel \Phi_t(v_1)\parallel^2} \Phi_t(v_1)\right).$$ By definition the two components of $\Theta_t$ are still orthogonal. If denotes $$\Theta_t=\left(\text{Proj}_1(\Theta_t),\text{Proj}_2(\Theta_t)\right),$$ then for every regular point $x\in M$ and every vector $v\in \mathcal{N}_x$, one has that $$\psi_t(v)=\text{Proj}_2(\Theta_t)(X(x),v).$$ By the continuity of $\Theta_t$, one can extend the definition of $\psi_t$ to singularities: for every vector $u\in\widetilde{\Lambda}$, one can define $$\widetilde{\mathcal{N}}_u=\{v\in T_{\pi(u)}M:~v\bot u\}.$$ Then, $\widetilde{\mathcal{N}}$ is a $(d-1)$-dimensional vector bundle on the base space $\widetilde{\Lambda}$. For every $u\in\widetilde{\Lambda}$, and every $v\in\widetilde{\mathcal{N}}_u$, one can define a flow $$\widetilde{\psi}_t(v)=\text{Proj}_2(\Theta_t)(u,v).$$ Then, the linear Poincar\'{e} flow $\psi_t$ can be \textquotedblleft embedded\textquotedblright in the flow $\Theta_t$. By the definition, $\text{Proj}_2(\Theta_t)$ is a continuous flow defined on $\widetilde{\Lambda}$. Thus, $\widetilde{\psi}_t(v)$ varies continuously with respect to the vector field $X$, the time $t$ and the vector $v$ and can be viewed as a compactification of $\psi_t$.   
   
   For every $x\in M\setminus{\rm Sing}(X)$ and every $\delta>0$ small, the \emph{\bf normal manifold} of $x$ is defined as $$N_x(\delta)={\rm exp}_x(\mathcal{N}_x(\delta)),$$ where $\mathcal{N}_x(\delta)=\left\{v\in\mathcal{N}_x:~\lvert v\rvert\leq\delta\right\}$. When $\delta$ is small enough, $N_x(\delta)$ is an imbedded submanifold which is diffeomorphic to $\mathcal{N}_x$. Furthermore, $N_x(\delta)$ is a local cross section transverse to the flow. To study the dynamics in a small neighborhood of a periodic orbit of a vector field, Poincar\'{e} defined the sectional return map of a cross section of a periodic point. By generalizing this idea to every regular point, one can define the sectional Poincar\'{e} map between any two cross sections at any two points in the same regular orbit. For every $T>0$ and $x\in M\setminus{\rm Sing}(X)$, the \emph{\bf Poincar\'e map} $$P_{x,\varphi_T(x)}:~N_x(\delta)\rightarrow N_{\varphi_T(x)}(\delta')$$ is the holonomy map generated by the flow $\varphi_t$, where $\delta$ and $\delta'$ depend on the choice of $x$ and $T$. And we call a local lift of Poincar\'e map $P_{x,\varphi_T(x)}$ in the normal bundle,  $$\mathcal{P}_{x,\varphi_T(x)}={\rm exp}^{-1}_{\varphi_T(x)}\circ P_{x,\varphi_T(x)}\circ {\rm exp}_x:~\mathcal{N}_x(\delta)\rightarrow\mathcal{N}_{\varphi_T(x)}(\delta'),$$ the \emph{\bf sectional Poincar\'e map}. Eliminate the influence of flow speed, we will also consider the \emph{\bf scaled sectional Poincar\'e map} $\mathcal{P}^*_{x,\varphi_T(x)}$ which is defined by $$\mathcal{P}^*_{x,\varphi_T(x)}=\dfrac{\mathcal{P}_{x,\varphi_T(x)}}{\lvert X(\varphi_T(x))\rvert}.$$ Moreover, we can see that $$D_x\mathcal{P}_{x,\varphi_T(x)}=\psi_T|\mathcal{N}_x:\mathcal{N}_x\rightarrow\mathcal{N}_{\varphi_T(x)},\qquad D_x\mathcal{P}^*_{x,\varphi_T(x)}=\psi^*_T|\mathcal{N}_x:\mathcal{N}_x\rightarrow\mathcal{N}_{\varphi_T(x)}.$$ 
     
   For the sectional Poincar\'e map, we have the following lemmas \cite{GY,WYZ} which explains the rationality of the definition, states the uniform continuity of the sectional Poincar\'e map up to flow speed, and estimates the return time.
     
\begin{Lemma}\label{well-def}{\rm(\cite[Lemma 2.3]{GY})} 
   Given $X\in\mathfrak{X}^1(M^d)$ and $T>0$, there exists $\beta_T>0$ such that for every $x\in M\setminus{\rm Sing}(X)$, the sectional Poincar\'e map $$\mathcal{P}_{x,\varphi_T(x)}:\mathcal{N}_x\left(\beta_T\lvert X(x)\rvert\right)\rightarrow \mathcal{N}_{\varphi_T(x)}\left(\beta_TC_T\lvert X(\varphi_T(x))\rvert\right)$$ is well defined, where $C_T=\lVert\psi_T\rVert$ is a bounded constant. 
\end{Lemma}     
     
\begin{Lemma}\label{uni-con}{\rm(\cite[Lemma 2.4]{GY})} 
   Given $X\in\mathfrak{X}^1(M^d)$ and $T>0$, and reducing $\beta_T>0$ in Lemma \ref{well-def} if necessary, for every $x\in M\setminus{\rm Sing}(X)$, for the sectional Poincar\'e map $$\mathcal{P}_{x,\varphi_T(x)}:~\mathcal{N}_x\left(\beta_T\lvert X(x)\rvert\right)\rightarrow\mathcal{N}_{\varphi_T(x)}\left(\beta_TC_T\lvert X(\varphi_T(x))\rvert\right),$$ $D\mathcal{P}_{x,\varphi_T(x)}$ is uniformly continuous in the following sense: for each $\epsilon>0$, there exists $\rho\in(0,\beta_T]$ such that for every $x\in M\setminus{\rm Sing}(X)$ and $y,~y'\in \mathcal{N}_x\left(\beta_T\lvert X(x)\rvert\right)$, if $\lvert y-y'\rvert\leq\rho\lvert X(x)\rvert$, then $$\lvert D_y\mathcal{P}_{x,\varphi_T(x)}-D_{y'}\mathcal{P}_{x,\varphi_T(x)}\rvert<\epsilon.$$  	
\end{Lemma}           
     
\begin{Lemma}\label{return}{\rm(\cite[Lemma 4.5]{WYZ})} 
   Given $X\in\mathfrak{X}^1(M^d)$, let $\beta=\beta_1$ and $C_1$ be the constants in Lemma \ref{well-def} associated to the time-one map $\varphi_1$. Then by reducing $\beta>0$ if necessary, there is $\kappa>0$ such that for every $x\in M\setminus{\rm Sing}(X)$ and every $y\in N_x\left(\beta\lvert X(x)\rvert\right)$, there is a unique $t=t(y)\in(0,2)$ satisfying that $$\varphi_t(y)\in N_{\varphi_1(x)}\left(\beta C_1\lvert X(\varphi_1(x))\rvert\right)\quad\text{and}\quad\lvert t(y)-1\rvert<\kappa d(x,y).$$   	
\end{Lemma}       

   Now, we review some definitions on various splitting of vector field.         
\begin{Definition}\label{DDS}
   Let $\Lambda$ be an invariant set of vector field $X$, $T_\Lambda M=E\oplus F$ be an invariant splitting with respect to the tangent flow $\Phi_t$ over $\Lambda$. We call the splitting $T_\Lambda M=E\oplus F$ a {\bf dominated splitting} with respect to the tangent flow $\Phi_t$, if there are constants $C\geq 1$, $\lambda>0$ such that for every $x\in\Lambda$ and $t\geq 0$, one has that $$\lVert\Phi_t\lvert_{E(x)}\rVert\cdot \lVert\Phi_{-t}\lvert_{F(\varphi_t(x))}\rVert\leq Ce^{-\lambda t}.$$ For simplicity, denete by $E\prec F$.
\end{Definition}

\begin{Definition}\label{FHS}
   Let $\Lambda$ be an invariant set of vector field $X$, $T_\Lambda M=E\oplus\langle X\rangle\oplus F$ be an invariant splitting with respect to the tangent flow $\Phi_t$ over $\Lambda$. The splitting $T_\Lambda M=E\oplus\langle X\rangle\oplus F$ is called a {\bf hyperbolic splitting} with respect to the tangent flow $\Phi_t$, if there are constants $C\geq 1$ and $\lambda>0$ such that  $$\lVert\Phi_t\lvert_{E(x)}\rVert\leq Ce^{-\lambda t}\quad\text{and}\qquad\lVert\Phi_{-t}\lvert_{F(x)}\rVert\leq Ce^{-\lambda t},\text{ for every $x\in\Lambda$ and every $t\geq 0$},$$ where $\langle X\rangle$ denotes the subspace generated by the vector field. Correspondingly, the set $\Lambda$ is called hyperbolic. 
\end{Definition} 
   Note that $\langle X\rangle$ is zero subspace $\{0\}$ at any singularity $x\in M$, and is a $1$-dimensional subspace at any regular point $x\in M$. Since the constants $C$ and $\lambda$ are independent of $x\in\Lambda$ in the Definition \ref{FHS}, one can derive that, in any hyperbolic set, regular points can not accumulate on singularities. Moreover, it is easy to prove that the hyperbolic splitting $E(x)\oplus\langle X(x)\rangle\oplus F(x)$ various continuously in $x$, and if $\Lambda$ is hyperbolic, so is the closure $\overline{\Lambda}$. 

\begin{Remark} 
   Since linear Poincar\'{e} flow $\psi_t$, scaled linear Poincar\'{e} flow $\psi^*_t$, and extended linear Poincar\'{e} flow $\widetilde{\psi}_t$ are induced by tangent flow $\Phi_t$, we can define corresponding dominated {\rm (}hyperbolic{\rm)} splitting with respect to $\psi_t$, $\psi^*_t$ and $\widetilde{\psi}_t$. For example, let $\Lambda$ be an {\rm (}not necessarily compact{\rm)} invariant set of flow $\varphi_t$, an invariant splitting $\mathcal{N}_\Lambda=E\oplus F$ over the invariant set $\Lambda$ is called {\bf dominated {\rm (}hyperbolic{\rm)} splitting} with respect to the linear Poincar\'{e} flow $\psi_t$, if there are constants $C\geq 1$, $\lambda>0$ such that for every $x\in\Lambda$ and $t\geq 0$, one has that $$\lVert\psi_t\lvert_{E(x)}\rVert\cdot \lVert\psi_{-t}\lvert_{F(\varphi_t(x))}\rVert\leq Ce^{-\lambda t}\qquad{\rm (}\lVert\psi_t\lvert_{E(x)}\rVert\leq Ce^{-\lambda t}\quad\text{and}\quad\lVert\psi_{-t}\lvert_{F(x)}\rVert\leq Ce^{-\lambda t}{\rm)}.$$ 
\end{Remark}

   The notion of star system came up from the study of the famous stability conjecture. Recall that a classical theorem of Smale \cite{SS70} (for diffeomorphisms) and Pugh--Shub \cite{PS70} (for flows) states that Axiom A plus the no-cycle condition implies the $\Omega$-stability. Palis and Smale \cite{PaS70} conjectured that the converse also holds, which has been known as the $\Omega$-stability conjecture. In the study of the conjecture, Pliss, Liao and Ma\~{n}\'{e} noticed an important condition called (by Liao) the star condition. Indeed, the $\Omega$-stability implies the star condition easily (Franks \cite{FJ71} and Liao \cite{LST79}). Thus whether the star condition can give back Axiom A plus the no-cycle condition became a striking problem, raised by Liao \cite{LST81} and Ma\~{n}\'{e} \cite{MR82}. Especially, for flows, there are counterexamples if the flow has a singularity. For instance, the geometric Lorenz attractor \cite{Guc76}, which has a singularity, is a star flow but fails to satisfy Axiom A. 

\begin{Definition}\label{SVF}
   A vector field $X\in\mathfrak{X}^1(M^d)$ is called a {\bf star vector field}, if there is a $C^1$ neighborhood $\mathcal{U}$ of $X$ such that for every $Y\in\mathcal{U}$, all singularities and all periodic orbits of $\varphi^Y_t$ are hyperbolic.
\end{Definition} 

   The set of all star vector fields on $M$ is denoted by $\mathfrak{X}^*(M)$ which is endowed with $C^1$-topology.    
   
\subsection{Lyapunov Exponents of Vector Fields}   
   In this section, we introduce some ergodic theory for vector field. Given a vector field $X\in\mathfrak{X}^r(r\geq1)$, a measure $\mu$ is called \emph{\bf invariant} with respect to the flow $\varphi^X_t$ (a vector field $X$), if $\mu$ is an invariant measure of $\varphi_T$ for every $T\in\mathbb{R}$. Similarly, a measure $\mu$ is called \emph{\bf ergodic} with respect to the flow $\varphi_t$, if $\mu$ is an invariant ergodic measure of $\varphi_T$ for every $T\in\mathbb{R}$. The set of all invariant measures and all ergodic measures of vector field are denoted by $\mathcal{M}(X)$, $\mathcal{E}(X)$ respectively. Let $\mu$ be an invariant measure of flow $\varphi_t$, by the Oseledec Theorem \cite{OV}, for $\mu$-almost every $x$, there are a positive integer $k(x)\in[1,d]$, real numbers $\chi_1(x)<\chi_2(x)<\cdots<\chi_{k(x)}(x)$ and a measurable $\Phi_t$-invariant splitting $$T_xM=E_1(x)\oplus E_2(x)\oplus\cdots\oplus E_k(x)$$ such that $$\lim_{t\to\pm\infty}\dfrac{1}{t}\log\parallel\Phi_t(v) \parallel=\chi_i(x),~\forall~v\in E_i(x),~v\neq0,~i=1,2,\cdots,k(x).$$ The numbers $\chi_1(x),\cdots,\chi_{k(x)}(x)$ are called {\bf Lyapunov exponents} at point $x$ of $\Phi_t$ with respect to $\mu$ and the vector formed by these numbers (counted with multiplicity, endowed with the increasing order) is called {\bf Lyapunov vector} at point $x$ of $\Phi_t$ with respect to $\mu$. The {\bf index} of $\mu$, denoted by $\text{Ind}(\mu)$, is defined as $$\text{Ind}(\mu)\triangleq\sum_{\chi_i(x)<0}\text{dim}E_i(x).$$ If $\mu$ is an ergodic invariant measure, then these numbers $k(x),\chi_1(x),\cdots,\chi_{k(x)}(x)$ are constants. By the Poincar\'{e} Recurrence Theorem, for $\mu$-almost every $x\in M$, one has that $$\lim_{t\to\pm\infty}\dfrac{1}{t}\log\parallel\Phi_t|_{\langle X(x)\rangle}\parallel=0,$$ where $\langle X(x)\rangle$ is the $1$-dimensional subspace of $T_xM$ generated by the flow direction $X(x)$. Thus, it follows that there exists one zero Lyapunov exponent for $\Phi_t$ along the flow direction.
  
   For an ergodic invariant measure $\mu$ which is not concentrated on ${\rm Sing}(X)$, according to the definition of linear Poincar\'e flow $\psi_t:\mathcal{N}\rightarrow\mathcal{N}$ and the Oseledec Theorem \cite{OV}, for $\mu$-almost every $x$, there are a positive integer $k\in[1,d-1]$, real numbers $\chi_1<\chi_2<\cdots<\chi_k$ and a measurable $\psi_t$-invariant splitting (for simplicity, we omit the base point) $$\mathcal{N}=E_1\oplus E_2\oplus\cdots\oplus E_k$$ on normal bundle such that $$\lim_{t\to\pm\infty}\dfrac{1}{t}\log\parallel\psi_t(v) \parallel=\chi_i,~\forall~v\in E_i,~v\neq 0,~i=1,2,\cdots,k.$$ Accordingly, the numbers $\chi_1,\cdots,\lambda_k$ are called \emph{\bf Lyapunov exponents} of $\psi_t$ with respect to the ergodic measure $\mu$ and the vector formed by these numbers (counted with multiplicity, endowed with the increasing order) is called \emph{\bf Lyapunov vector} of $\psi_t$ with respect to $\mu$. According to the definition of the scaled linear Poincar\'e flow $\psi_t^*:\mathcal{N}\rightarrow\mathcal{N}$, we have that $$\psi^*_t(v)=\frac{\parallel X(x)\parallel}{\parallel X(\varphi_t(x))\parallel}\psi_t(v)=\frac{\psi_t(v)}{\parallel\Phi_t|_{\langle X(x)\rangle}\parallel},\quad\text{for each non-zero }v\in\mathcal{N}.$$ Since the Lyapunov exponent for $\Phi_t$ along the flow direction is zero, namely,  $$\lim_{t\to\pm\infty}\dfrac{1}{t}\log\parallel\Phi_t|_{\langle X(x)\rangle}\parallel=0,$$ one has that  $$\lim_{t\to\pm\infty}\dfrac{1}{t}\log\parallel\psi^*_t(v)\parallel=\lim_{t\to\pm\infty}\dfrac{1}{t}\left(\log\parallel\psi_t(v)\parallel-\log\parallel\Phi_t|_{\langle X(x)\rangle}\parallel\right)=\lim_{t\to\pm\infty}\dfrac{1}{t}\log\parallel\psi_t(v)\parallel.$$ Thus, for $\mu$-almost every $x$, we also have a positive integer $k\in[1,d-1]$, real numbers $\chi_1<\chi_2<\cdots<\chi_k$ and a measurable $\psi^*_t$-invariant splitting (for simplicity, we omit the base point) $$\mathcal{N}=E_1\oplus E_2\oplus\cdots\oplus E_k$$ on normal bundle satisfying that $$\lim_{t\to\pm\infty}\dfrac{1}{t}\log\parallel\psi^*_t(v) \parallel=\chi_i,~\forall~v\in E_i,~v\neq 0,~i=1,2,\cdots,k.$$ Then, the numbers $\chi_1,\cdots,\chi_k$ are called \emph{\bf Lyapunov exponents} of $\psi^*_t$ with respect to the ergodic measure $\mu$ and the vector formed by these numbers (counted with multiplicity, endowed with the increasing order) is called \emph{\bf Lyapunov vector} of $\psi^*_t$ with respect to $\mu$. It means that the Lyapunov exponents of the scaled linear Poincar\'{e} flow and those of the linear Poincar\'{e} flow are the same. Hence, the scaled linear Poincar\'{e} flow and the linear Poincar\'{e} flow also have the same Oseledec splitting.

   Given a periodic point $z$ (period $\pi(z)$) of vector field $X$, we can define the Borel probability invariant measure $$\mu_z=\dfrac{1}{\pi(z)}\int_{0}^{\pi(z)}\delta_{\varphi_t(z)}dt$$ supported on the periodic orbit of $z$, where $\delta_y$ denotes the Dirac measure supported on $y$. By the Oseledec Theorem, we also have corresponding Lyapunov exponents and Lyapunov vector with respect to the invariant measure $\mu_z$.  
         
\begin{Definition}\label{Def:hyperbolicmeasure}
   An ergodic measure $\mu$ of the flow $\varphi_t$ is \emph{regular} if it is not supported on a singularity. A regular ergodic measure is \emph{hyperbolic}, if the Lyapunov exponents of the \emph{linear Poincar\'e flow} $\psi_t$ are all non-zero.
\end{Definition}
     
\begin{Remark}
   We can also define the hyperbolicity of an ergodic measure by using the tangent flow $\Phi_t={\rm d}\varphi_t$ as usual. However, for ergodic measures that are not supported on singularities, there will be one zero Lyapunov exponent for the tangent flow along the flow direction. 
\end{Remark}

   Given a hyperbolic ergodic measure $\mu$, let $\Gamma$ (called the Oseledec's basin of $\mu$ in \cite{WCZ21}) be the set of points which are regular with respect to linear Poincar\'{e} $\psi_t$ in the sense of Oseledec \cite{OV}. For $x\in\Gamma$, let $$\mathcal{N}_x=E_1(x)\oplus E_2(x)\oplus\cdots\oplus E_s(x)\oplus E_{s+1}(x)\oplus\cdots\oplus E_k(x)~~(k\in[0,d-1]\text{ is an integer})$$ be the decomposition of normal bundle corresponding to the distinct Lyapunov exponents $$\chi_1(\mu)<\chi_2(\mu)<\cdots<\chi_s(\mu)<0<\chi_{s+1}(\mu)<\cdots<\chi_k(\mu)$$ with multiplicities $d_1,d_2,\cdots,d_k\geq 1$. Then, dim$(E_i(x))=d_i$, for $i=1,2,\cdots,k$. Denote $$E^s=E_1\oplus E_2\oplus\cdots\oplus E_s, \text{ and } E^u=E_{s+1}\oplus E_{s+2}\oplus\cdots\oplus E_k.$$ We call the splitting $$\mathcal{N}_\Gamma=E^s\oplus E^u$$ {\bf hyperbolic Oseledec splitting}. Correspondingly, we call $E^s$ and $E^u$ stable bundle and unstable bundle. 

   For star vector, it is proved in \cite{SGW14} that every ergodic measure of a star vector field $X\in\mathfrak{X}^*(M)$ is hyperbolic. Li, Shi, Wang and Wang \cite{LSWW20} asserts (in the following Lemma \ref{SWDS}) the existence of dominated splitting on $\mathcal{N}_{\text{supp}(\mu)\backslash\text{Sing}(X)}$ for each ergodic measure $\mu$ of a star vector field with respect to the scaled linear Poincar\'{e} flow.
   
\begin{Lemma}\label{SWDS}{\rm(\cite[Lemma 2.6]{LSWW20})} 
   If $\mu$ is a nontrivial ergodic measure of a star vector field $X\in\mathfrak{X}^*(M)$, then there is a dominated splitting $\mathcal{N}_{\text{supp}(\mu)\backslash\text{Sing}(X)}=E\oplus F$ with respect to the scaled linear Poincar\'{e} flow $\psi^*_t$ such that $$\text{\rm dim}(E)=\text{\rm Ind}(\mu).$$  
\end{Lemma}

\section{Lyapunov norms and Shadowing lemma}

\subsection{Lyapunov norms and Lyapunov charts}
   In this section, we only consider $C^1$ vector field $X\in\mathfrak{X}^1(M^d)$. Since we can not define the linear Poincar\'e flow (see Definition \ref{Def:linearPoincare}) at singularities, we lose some uniformity. Fortunately for us, we may still regain some uniform properties when consider the scaled linear Poincar\'e flow.
     
   On the tangent space $TM$, the \emph{Sasaki metric}(see \cite[ Subsection 2.1]{BMW12}), denote by $d_{S}$, is a useful metric which is introduced by the Riemannian metric on $M$. By the compactness of $M$, there is a small constant $\rho_S$ such that for every $x,~y\in M$ with $d(x,y)<\rho_S$, there is a unique geodesic joining $y$ to $x$. Furthermore, if $u_x\in T_xM$ and $u_y\in T_yM$ with $d_{S}(u_x,u_y)\leq\rho_S$ or $\lvert u_x-u'_y\rvert\leq\rho_S$, where $u'_y\in T_xM$ is obtained by parallel transporting $u_y$ along the geodesic from $y$ to $x$, then $$\left(d(x,y)+\lvert u_x-u'_y\rvert\right)/2\leq d_{S}(u_x,u_y)\leq 2\left(d(x,y)+\lvert u_x-u'_y\rvert\right).$$ Moreover, there is a real number $K_{G^1}$ such that for any unit vectors $u_1,~u_2\in TM$, one has that $$d_{G^1}(\langle u_1\rangle,\langle u_2\rangle)\leq K_{G^1}\cdot d_{S}(u_1,u_2),$$ where $d_{G^1}$ is the Grassman distance on the $1$-dimensional Grassmann bundle over $M$, i.e., $$G^1=\left\{N\subset T_xM: N \text{ is a $1$ dimensional linear subspace of $T_xM$}, x\in M\right\}.$$ Since $M$ is compact, there is a constant $K_0$ such that $$\max_{x\in M}\left\{\lvert X(x)\rvert,~\lVert DX(x)\rVert\right\}\leq K_0.$$ Li, Liang, and Liu pointed out the relationship between the Sasaki metric and the Riemannian metric as the following Lemma \ref{Dlgx}.
       
\begin{Lemma}\label{Dlgx}{\rm\cite[Lemma 3.1]{LLL24}}
   There are constants $K_1>K_0$ and $\beta^*_0>0$ such that for every $x\in M\backslash\text{Sing}(X)$ and each $y\in B(x,\beta^*_0\lvert X(x)\rvert)$, one has that
\begin{itemize}
   \item[{\rm (1)}] $1-K_1\cdot \dfrac{d(x,y)}{\lvert X(x)\rvert}\leq\dfrac{\lvert X(x)\rvert}{\lvert X(y)\rvert}\leq1+K_1\cdot \dfrac{d(x,y)}{\lvert X(x)\rvert}$;
		
   \item [{\rm (2)}] $d_S\left(\dfrac{X(x)}{\lvert X(x)\rvert},\dfrac{X(y)}{\lvert X(y)\rvert}\right)\leq K_1\cdot\dfrac{d(x,y)}{\lvert X(x)\rvert}$;
		
   \item [{\rm (3)}] $d_{G^1}\left(\langle X(x)\rangle,\langle X(y)\rangle\right)\leq K_1\cdot\dfrac{d(x,y)}{\lvert X(x)\rvert}$; 
\end{itemize}   	
\end{Lemma}

   Let $r_0$ be the radius such that the exponential map is a $C^{\infty}$ diffeomorphism in $T_xM(r_0)$. Reducing $\beta^*_0>0$ if necessary, we may assume that $$10\beta^*_0K_1<\min\{r_0,1\}.$$ Thus, the item $(1)$ of Lemma \ref{Dlgx} means that there is no singularity in $B(x,\beta^*_0\lvert X(x)\rvert)$. Next, we analyze the sectional Poincar\'{e} flow in local charts of scaled neighborhoods, the so-called \emph{Liao scaled charts}, to obtain some uniformity. One can see similar discussions in \cite{CY2017,GY,WW19}.
    
   Let $\{e_1,\cdots,e_d\}$ be an orthonormal basis of $\mathbb{R}^d$. For any $x\in M^d$, take an orthonormal basis $\{e^x_1,\cdots,e^x_d\}$ of $T_xM^d$, we can get a linear iosmetry $$C_x:\mathbb{R}^d\mapsto T_xM^d,$$ which is a metric change of coordinates from Euclidean space to tangent space, such that $$C_x(e_i)=e^x_1,\text{ for $i=1,2,\cdots,d$}.$$ Then, $$\langle C_x(u), C_x(v)\rangle_x=\langle u,v\rangle,~ \forall~u,v\in{R}^d,$$ where $\langle\cdot,\cdot\rangle$ is the standard scalar product on $\mathbb{R}^d$. Denote by $${\rm Exp}_x={\rm exp}_x\circ C_x:\mathbb{R}^d\mapsto M^d,~~\mathbb{R}^d(r)=\left\{v\in\mathbb{R}^d:~\lvert v\rvert\leq r\right\}.$$ Then, ${\rm Exp}_x|_{\mathbb{R}^d(r_0)}$ is a $C^{\infty}$ diffeomorphism from $\mathbb{R}^d(r_0)$ to $B(x,r_0)$, for every $x\in M$. The flow in the neighborhood $B(x,r_0)$ can be locally lifted to $\mathbb{R}^d(r_0)$ by ${\rm Exp}_x$. Namely, if there exists a point $p\in\mathbb{R}^d$ and $t_1<0<t_2$, such that $$\varphi_t({\rm Exp}_x(p))\in B(x,r_0),~\text{ for every $t\in[t_1,t_2]$},$$ then we can define a flow $$\widetilde{\varphi_{x,t}}(p)=\left({\rm Exp}_x|_{\mathbb{R}^d(r_0)}\right)^{-1}\circ\varphi_t\circ{\rm Exp}_x(p),~\text{ for every $t\in[t_1,t_2]$}.$$ In this Liao scaled chart of $x$, the flow $\varphi^X_t$ generated by the vector field $X$ satisfying the differential equation $$\dfrac{dz}{dt}=\widehat{X}_x(z),$$ where $\widehat{X}_x(z)=D\left({\rm Exp}_x|_{\mathbb{R}^d(r_0)}\right)^{-1}\circ X\circ{\rm Exp}_x(z)$ and ${\rm Exp}_x(0)=x$, $\widehat{X}_x(0)=X(0)$. Since $M$ is compact, there is a constant $K_{{\rm E}}>1$ such that $$\max_{x\in M^d}\left\{\left\lVert D{\rm Exp}_x|_{\mathbb{R}^d(r_0)}\right\rVert,~\left\lVert D\left({\rm Exp}_x|_{\mathbb{R}^d(r_0)}\right)^{-1}\right\rVert\right\}<K_{{\rm E}}.$$ Thus, increasing $K_0$ if necessary, we may also assume that $$\max_{x\in M}\left\{\max_{p\in\mathbb{R}^d(r_0)}\left\{\left\lvert\widehat{X}_x(p)\right\rvert,~\left\lVert D\widehat{X}_x(p)\right\rVert\right\}\right\}<K_0.$$ Therefore, $$\left\lvert\widehat{X}_x(p)-\widehat{X}_x(0)\right\rvert<K_0\lvert p\rvert\leq K_0 r_0.$$  

   For every regular point $x\in M$, take $e^x_1=\dfrac{X(x)}{\lvert X(x)\rvert}$. Then, $\widehat{X}_x(0)=X(0)=\left(\lvert X(x)\rvert,0,0,\cdots,0\right)$. Given $x\in M\backslash{\rm Sing}(X)$, for every $y\in B(x,\beta^*_0\lvert X(x)\rvert)$, we may choose an orthonormal basis $\{e^y_1,\cdots,e^y_d\}$ of $T_yM$ such that $e^y_1=\dfrac{X(y)}{\lvert X(y)\rvert}$. Therefore, by Lemma \ref{Dlgx}, for any $y_1,~y_2\in B(x,\beta^*_0\lvert X(x)\rvert)$, we have that $$d_S\left(e^{y_1}_j,e^{y_2}_j\right)<2K_1\dfrac{d(x,y)}{\lvert X(x)\rvert},\text{ for every $j=1,2,\cdots,d$}.$$ Under this orthonormal bases in $T_{B(x,\beta^*_0\lvert X(x)\rvert)}M$, we can get uniform estimate on $C^1$ norm of the diffeomorphism ${\rm Exp}_x$ as the following Lemma \ref{Eugx}.   
   
\begin{Lemma}\label{Eugx}{\rm\cite[Lemma 3.2]{LLL24}}
   There is $K_2>1$ such that for every $x\in M\backslash{\rm Sing}(X)$, one has that $$\left\lVert\left({\rm Exp}^{-1}_{y_1}\circ{\rm Exp}_{y_2}-id_{\mathbb{R}^d}\right)\big\lvert_{\mathbb{R}^d(r_0/2)}\right\rVert_{C^1}\leq K_2\dfrac{d(y_1,y_2)}{\lvert X(x)\rvert},$$ for every $y_1,~y_2\in B(x,\beta^*_0\lvert X(x)\rvert)$, where $\lVert\cdot\rVert$ is the $C^1$ norm.	
\end{Lemma}

\subsection{Shadowing lemma}
   The first time \emph{shadowing} appeared in the literature was in a paper \cite{Si72} by Sina\u{\i} where it is shown that Anosov diffeomorphisms have shadowing and furthermore that every pseudo-orbit has a unique point shadowing it. Shadowing describes the situation where a true orbit of a dynamical system lies uniformly near a pseudo-orbit. The shadowing lemma for diffeomorphisms has been described perfectly. But for flows, the shadowing lemma is rather more complicated than that for diffeomorphisms.
   
\begin{Definition}
   Let $\Lambda$ be an invariant set and $E\subset\mathcal{N}_{\Lambda\backslash\text{Sing}(X)}$ an invariant subbundle of the linear Poincar\'{e} flow $\psi_t$. For $C>0$, $\eta>0$ and $T>0$, $x\in\Lambda\backslash\text{Sing}(X)$ is called $(C,\eta,T,E)$-$\psi^*_t$-contracting if there exists a time partition: $0=t_0<t_1<\cdots<t_n<\cdots$ such that $t_{i+1}-t_i\leq T$ for each $i\in\mathbb{N}$, and $t_n\rightarrow\infty$ as $n\rightarrow\infty$, and $$\prod_{i=0}^{n-1}\left\lVert\psi^*_{t_{i+1}-t_i}|_{E(\varphi_{t_i}(x))}\right\rVert\leq Ce^{-\eta t_n},~\text{for every }n\in\mathbb{N}.$$ A point $x\in\Lambda\backslash\text{Sing}(X)$ is called $(C,\eta,T,E)$-$\psi^*_t$-expanding if it is $(C,\eta,T,E)$-$\psi^*_t$-contracting for the vector field $-X$. 	
\end{Definition}

\begin{Definition}
   Given $\eta>0$, $T>0$. For any $x\in M\backslash\text{Sing}(X)$, $T_0>T$, the orbit arc $\varphi_{[0,T_0]}$ is called $(\eta,T)$-$\psi^*_t$-quasi hyperbolic with respect to a direct sum splitting $\mathcal{N}_x=E_x\oplus F_x$, if there is a time partition: $$0=t_0<t_1<\cdots<t_k=T_0,\quad t_{i+1}-t_i\leq T,$$ such that for $n=0,1,\cdots,k-1$, one has that $$\prod_{i=0}^{n-1}\left\lVert\psi^*_{t_{i+1}-t_i}|_{E(\varphi_{t_i}(x))}\right\rVert\leq e^{-\eta t_k},\quad\prod_{i=n}^{k-1}m\left(\psi^*_{t_{i+1}-t_i}|_{F(\varphi_{t_i}(x))}\right)\geq e^{\eta(t_k-t_n)},$$ $$\dfrac{\left\lVert\psi^*_{t_{n+1}-t_n}|_{E(\varphi_{t_n}(x))}\right\rVert}{m\left(\psi^*_{t_{n+1}-t_n}|_{F(\varphi_{t_n}(x))}\right)}\leq e^{-\eta(t_{n+1}-t_n)}.$$      	
\end{Definition}

   The following shadowing lemma for singular flows was given by Liao \cite{LST85}. 
\begin{Theorem}\label{Shadow}
   Let $X\in\mathfrak{X}^1(M^d)$, $\Lambda\in M\backslash\text{Sing}(X)$ be an invariant set with a dominated splitting $\mathcal{N}_\Lambda=E\oplus F$ and $\eta>0,~T>0$. For $\alpha>0$ and $\varepsilon>0$, there is $\delta>0$ such that for every $(\eta,T)$-$\psi^*_t$-quasi hyperbolic orbit segment $\varphi_{[0,T_0]}(x)$ satisfying
\begin{itemize}
   \item $d(x,\text{Sing}(X))>\alpha$ and $d(\varphi_{T_0}(x),\text{Sing}(X))>\alpha$;
   
   \item $x\in\Lambda$, $\varphi_{T_0}(x)\in\Lambda$ and $d(x,\varphi_{T_0}(x))<\delta$; 
\end{itemize}  
   there exists a $C^1$ strictly increasing function $\theta:[0,T_0]\rightarrow\mathbb{R}$ and a periodic point $p\in M$  such that
\begin{description}
   \item[(1)] $\theta(0)=0$ and $1-\varepsilon<\theta'(t)<1+\varepsilon$, for every $t\in[0,T_0]$;
   
   \item[(2)] $p$ is periodic: $\varphi_{\theta(T_0)}(p)=p$;
   
   \item[(3)] $d(\varphi_t(x),\varphi_{\theta(t)}(p))<\varepsilon\lvert X(\varphi_t(x))\rvert$, for every $t\in[0,T_0]$.
\end{description}    
\end{Theorem}

\section{ Density of periodic measures: proof of Theorem \ref{ThmA}}

   In this section, we give the proof of Theorem \ref{ThmA}. Before proving the Theorem \ref{ThmA}, we give a discussion about the Oseledec splitting. Let $\mu$ be a hyperbolic ergodic invariant measure of the flow $\varphi_t$ generated by a $C^1$ vector field $X\in\mathfrak{X}^1(M^3)$ on three-dimensional manifolds, $\lambda_1\leq\lambda_2$ be the Lyapunov exponents of $\psi^*_t$ with respect to the hyperbolic ergodic invariant measure $\mu$. Then, there is a measurable $\psi^*_t$-invariant splitting $$\mathcal{N}_\Gamma=E_1\oplus E_2$$ of normal bundle on $\Gamma$ (Oseledec's basin of $\mu$) satisfying that $$\lim_{t\to\pm\infty}\dfrac{1}{t}\log\parallel\psi^*_t(v) \parallel=\lambda_i,~\forall~v\in E_i,~v\neq 0,~i=1,2.$$   
   
\begin{Lemma}\label{time}
   If the Oseledec splitting $$\mathcal{N}_\Gamma=E_1\oplus E_2$$ of a hyperbolic ergodic invariant measure $\mu$ with respect to the three-dimensional flow $\varphi_t$ is a dominated splitting, then for every $0<\epsilon\ll\min\{\lvert\lambda_1\rvert,\lvert\lambda_2\rvert\}$ and every $\delta\in(0,1)$, there is a positive number $L=L(\epsilon,\delta)$ such that
\begin{itemize}
   \item for every $T\geq L$, there exists a measurable set $\Gamma^T=\Gamma^T(\epsilon,\delta)\subset\Gamma$ with $\mu(\Gamma^T)\geq 1-\delta$;
   
   \item there is a natural number $N=N(T)$ such that for every $n\geq N$ and every $x\in\Gamma^T$, we have that $$e^{(\lambda_j-\epsilon)nT}\leq\prod_{i=0}^{n-1}m\left(\psi^*_T|_{E_j(\varphi_{iT}(x))}\right)\leq\prod_{i=0}^{n-1}\left\lVert\psi^*_T|_{E_j(\varphi_{iT}(x))}\right\rVert\leq e^{(\lambda_j+\epsilon)nT},~j=1,2.$$
\end{itemize}     	
\end{Lemma}

\begin{proof}
   According to the Oseledec Theorem, for every $x\in\Gamma$, we have that
\begin{equation*}
   \lim_{t\rightarrow+\infty}\dfrac{1}{t}\log m\left(\psi^*_t|_{E_j(x)}\right)=\lim_{t\rightarrow+\infty}\dfrac{1}{t}\log\left\lVert\psi^*_t|_{E_j(x)}\right\rVert=\lambda_j,~j=1,2\tag{\textasteriskcentered}
\end{equation*}   
   Now, for every $0<\varepsilon\ll\min\{\lvert\lambda_1\rvert,\lvert\lambda_2\rvert\}$ and every $\delta\in(0,1)$, we will complete the proof of the lemma in three steps.
      
   {\bf Step $1$:} By equation $(*)$, for each $j=1,2$, we have that $$\lim_{t\rightarrow+\infty}\left\lvert\dfrac{1}{t}\log\left\lVert\psi^*_t|_{E_j(x)}\right\rVert-\lambda_j\right\rvert=0,~~\text{ for every }x\in\Gamma.$$ Therefore, $$\lim_{t\rightarrow+\infty}\int_M\left\lvert\dfrac{1}{t}\log\left\lVert\psi^*_t|_{E_j(x)}\right\rVert-\lambda_j\right\rvert d\mu(x)=0.$$ Thus, there exists a positive number $L^+_j$ such that for every $T\geq L^+_j$, we have that $$\int_M\left\lvert\dfrac{1}{T}\log\left\lVert\psi^*_T|_{E_j(x)}\right\rVert-\lambda_j\right\rvert d\mu(x)\leq\dfrac{\delta\epsilon}{36}.$$ Since Oseledec splitting  $$\mathcal{N}_\Gamma=E_1\oplus E_2$$ is a dominated splitting, Oseledec splitting is a continuous splitting \cite{BDV05}. Therefore, for every fix $j\in\{1,2\}$ and natural number $T$, the function $$f_{j,T}(x):=\left\lvert\dfrac{1}{T}\log\left\lVert\psi^*_T|_{E_j(x)}\right\rVert-\lambda_j\right\rvert$$ is continuous on $\Gamma$. On the other hand, $\mu$ is invariant measure of $\varphi_T$. By Birkhoff Ergodic theorem, there exists a measurable function $g_{j,T}$ such that for $\mu$-almost every $x\in M$, we have that $$\lim_{n\rightarrow+\infty}\dfrac{1}{n}\sum_{i=0}^{n-1}f_{j,T}(\varphi_{iT})=\lim_{n\rightarrow+\infty}\dfrac{1}{nT}\sum_{i=0}^{n-1}\left\lvert\log\left\lVert\psi^*_T|_{E_j(\varphi_{iT}(x))}\right\rVert-T\lambda_j\right\rvert=g_{j,T}(x),$$ and
\begin{equation*}
   \int_Mf_{j,T}(x)d\mu(x)=\int_M\left\lvert\dfrac{1}{T}\log\left\lVert\psi^*_T|_{E_j(x)}\right\rVert-\lambda_j\right\rvert d\mu(x)=\int_Mg_{j,T}(x)d\mu(x)\leq\dfrac{\delta\epsilon}{36}.\tag{\textasteriskcentered\textasteriskcentered}
\end{equation*} 	 
   Set $$\Gamma^+_{j,T}=\left\{x\in\Gamma:~g_{j,T}(x)\leq\dfrac{\epsilon}{2}\right\}.$$ By inequality $(**)$, we have that $$\mu\left(\Gamma^+_{j,T}\right)\geq 1-\dfrac{\delta}{18}$$ and $$\lim_{n\rightarrow+\infty}\dfrac{1}{nT}\sum_{i=0}^{n-1}\left\lvert\log\left\lVert\psi^*_T|_{E_j(\varphi_{iT}(x))}\right\rVert-T\lambda_j\right\rvert\leq\dfrac{\epsilon}{2},\quad\text{ for every }x\in\Gamma^+_{j,T}.$$ Namely, for every $x\in\Gamma^+_{j,T}$, there exists $N^+_j(x)$ such that $$\prod_{i=0}^{n-1}\left\lVert\psi^*_T|_{E_j(\varphi_{iT}(x))}\right\rVert\leq e^{(\lambda_j+\epsilon)nT},\quad\text{ for every }n\geq N^+_j(x).$$

   {\bf Step $2$:} By equation $(*)$, for each $j=1,2$, we have that $$\lim_{t\rightarrow+\infty}\left\lvert\dfrac{1}{t}\log m\left(\psi^*_t|_{E_j(x)}\right)-\lambda_j\right\rvert=0,\quad\text{ for every }x\in\Gamma.$$ Therefore, $$\lim_{t\rightarrow+\infty}\int_M\left\lvert\dfrac{1}{t}\log m\left(\psi^*_t|_{E_j(x)}\right)-\lambda_j\right\rvert d\mu(x)=0.$$ Thus, there exists a positive number $L^-_j$ such that for every $T\geq L^-_j$, we have that $$\int_M\left\lvert\dfrac{1}{T}\log m\left(\psi^*_T|_{E_j(x)}\right)-\lambda_j\right\rvert d\mu(x)\leq\dfrac{\delta\epsilon}{36}.$$ Since Oseledec splitting $$\mathcal{N}_\Gamma=E_1\oplus E_2$$ is a dominated splitting, Oseledec splitting is a continuous splitting \cite{BDV05}. Therefore, for every fix $j\in\{1,2\}$ and natural number $T$, the function $$\widetilde{f}_{j,T}(x):=\left\lvert\dfrac{1}{T}\log m\left(\psi^*_T|_{E_j(x)}\right)-\lambda_j\right\rvert$$ is continuous on $\Gamma$. On the other hand, $\mu$ is invariant measure of $\varphi_T$. By Birkhoff Ergodic theorem, there exists a measurable function $\widetilde{g}_{j,l}$ such that for $\mu$-almost every $x\in M$, we have that $$\lim_{n\rightarrow+\infty}\dfrac{1}{n}\sum_{i=0}^{n-1}\widetilde{f}_{j,T}(\varphi_{iT})=\lim_{n\rightarrow+\infty}\dfrac{1}{nT}\sum_{i=0}^{n-1}\left\lvert\log m\left(\psi^*_T|_{E_j(\varphi_{iT}(x))}\right)-T\lambda_j\right\rvert=\widetilde{g}_{j,T}(x),$$ and
\begin{equation*}
   \int_M\widetilde{f}_{j,T}(x)d\mu(x)=\int_M\left\lvert\dfrac{1}{T}\log m\left(\psi^*_T|_{E_j(x)}\right)-\lambda_j\right\rvert d\mu(x)=\int_M\widetilde{g}_{j,T}(x)d\mu(x)\leq\dfrac{\delta\epsilon}{36}.\tag{\textasteriskcentered\textasteriskcentered\textasteriskcentered}
\end{equation*}   
   Set $$\Gamma^-_{j,T}=\left\{x\in\Gamma:~\widetilde{g}_{j,T}(x)\leq\dfrac{\epsilon}{2}\right\}.$$ By inequality $(***)$, we have that $$\mu\left(\Gamma^-_{j,T}\right)\geq 1-\dfrac{\delta}{18}$$ and $$\lim_{n\rightarrow+\infty}\dfrac{1}{nT}\sum_{i=0}^{n-1}\left\lvert\log m\left(\psi^*_T|_{E_j(\varphi_{iT}(x))}\right)-T\lambda_j\right\rvert\leq\dfrac{\epsilon}{2},\text{ for every }x\in\Gamma^-_{j,T}.$$ Namely, for every $x\in\Gamma^-_{j,l}$, there exists $N^-_j(x)$ such that $$\prod_{i=0}^{n-1}m\left(\psi^*_T|_{E_j(\varphi_{iT}(x))}\right)\geq e^{(\lambda_j-\varepsilon)nT},\quad\text{ for every }n\geq N^-_j(x).$$  
	
   {\bf Step $3$:} Set $$L:=\max\{L^+_1,~L^-_1,~L^+_2,~L^-_2\},\quad\Lambda^T:=\bigcap_{j=1}^2\left(\Gamma^+_{j,T}\cap\Gamma^-_{j,T}\right),\quad\text{ for every }T\geq L.$$ Then, $\Lambda^T\subset\Gamma$ and $$\mu\left(\Lambda^T\right)>1-\dfrac{\delta}{2}.$$ Therefore, for every $x\in\Lambda^T$, there exists $N(x)=\max\{N^+_1(x),~N^-_1(x),~N^+_2(x),~N^-_2(x)\}$ such that for every $j\in\{1,2\}$, one has that
\begin{equation*}
   e^{(\lambda_j-\epsilon)nT}\leq\prod_{i=0}^{n-1}m\left(\psi^*_T|_{E_j(\varphi_{iT}(x))}\right)\leq\prod_{i=0}^{n-1}\left\lVert\psi^*_T|_{E_j(\varphi_{iT}(x))}\right\rVert\leq e^{(\lambda_j+\epsilon)nT},\text{ for every }n\geq N(x).\tag{\dag}
\end{equation*}    
   
   For $x\in\Lambda^T$, let $$\widetilde{N}(x):=\min\left\{N(x):~\text{ the inequality $(\dag)$ holds for all } n\geq N(x)\right\}.$$ Denote by $$\Lambda^T_n=\{x\in\Lambda^T:~\widetilde{N}(x)\leq n\}.$$ Then, $\Lambda^T_n\subset\Lambda^T_{n+1}$ and $$\Lambda^T=\bigcup_{n=1}^{+\infty}\Lambda^T_n.$$ Consequently, there is $n_0$ satisfying $\mu\left(\Lambda^T_{n_0}\right)>1-\delta$. For proving this lemma, we set $$N=N(T)=n_0,\quad\Gamma^T=\Lambda^T_{n_0}.$$
\end{proof}

\begin{proof}[Proof of Theorem \ref{ThmA}]
   Let $\mu$ be an ergodic hyperbolic invariant measure of the flow $\varphi_t$ generated by the star vector field $X\in\mathfrak{X}^*(M)$ on three-dimensional manifolds, $\lambda_1\leq\lambda_2$ be the Lyapunov exponents of $\psi^*_t$ with respect to the ergodic hyperbolic invariant measure $\mu$. Then, there is a measurable $\psi^*_t$-invariant splitting $$\mathcal{N}_\Gamma=E_1\oplus E_2$$ of normal bundle on $\Gamma$ (Oseledec's basin of $\mu$) satisfying that $$\lim_{t\to\pm\infty}\dfrac{1}{t}\log\parallel\psi^*_t(v) \parallel=\lambda_j,~\forall~v\in E_j,~v\neq 0,~j=1,2.$$ By Lemma \ref{SWDS}, there is a dominated splitting $\mathcal{N}_\Gamma=\widetilde{E_1}\oplus\widetilde{E_2}$ with respect to the scaled linear Poincar\'{e} flow $\psi^*_t$ such that $$\text{\rm dim}(\widetilde{E_1})=\text{\rm Ind}(\mu).$$ Regarding the Lyapunov exponents, we will discuss it in three cases.
    
   {\bf Case (1)}: If $\lambda_1$ and $\lambda_2$ are both negative , then $\text{\rm Ind}(\mu)=2$. Therefore, $\text{\rm dim}(\widetilde{E_1})=2$. Since $\text{\rm dim}(M)=3$, the fact that $\text{\rm dim}(\widetilde{E_1})=2$ contradicts that $\mathcal{N}_\Gamma=\widetilde{E_1}\oplus\widetilde{E_2}$ is a dominated splitting. This case is not valid. 
   
   {\bf Case (2)}: If $\lambda_1$ and $\lambda_2$ are both negative, then $\text{\rm Ind}(\mu)=0$. By similar discussion as {\bf Case (1)}, this case is also not valid.          
   
   {\bf Case (3)}: If $\lambda_1<0<\lambda_2$, then we claim that $\mathcal{N}_\Gamma=E_1\oplus E_2$ is a dominated splitting.     
\begin{claim}
   The Oseledec splitting $\mathcal{N}_\Gamma=E_1\oplus E_2$ is a dominated splitting.
\end{claim}
\begin{proof}[Proof of Claim]
   Since $\mathcal{N}_\Gamma=\widetilde{E_1}\oplus\widetilde{E_2}$ is a dominated splitting with respect to the scaled linear Poincar\'{e} flow $\psi^*_t$ such that $\text{\rm dim}(\widetilde{E_1})=\text{\rm Ind}(\mu)=1$, we only need to show that the two splitting $\mathcal{N}_\Gamma=E_1\oplus E_2$ and $\mathcal{N}_\Gamma=\widetilde{E_1}\oplus\widetilde{E_2}$ coincide.
   
   Given $x\in\Gamma$, for the splitting $\mathcal{N}_x=E_1(x)\oplus E_2(x)$, fix $\epsilon\in\left(0,\dfrac{\min\{\lvert\lambda_1\rvert,~\lambda_2\}}{8}\right)$ small, there is a natural number $T_0>0$ such that for any $t\geq T_0$, one has that $$e^{(\lambda_1-\epsilon)t}\lvert v\rvert\leq\parallel\psi^*_t(v)\parallel\leq e^{(\lambda_1+\epsilon)t}\lvert v\rvert,\quad\text{ for any nonzero vector }v\in E_1(x),\text{ and }$$ $$e^{(\lambda_2-\epsilon)t}\lvert v\rvert\leq\parallel\psi^*_t(v)\parallel\leq e^{(\lambda_2+\epsilon)t}\lvert v\rvert,\quad\text{ for any nonzero vector }v\in E_2(x).$$ For the splitting $\mathcal{N}_x=\widetilde{E_1}(x)\oplus\widetilde{E_2}(x)$, by the definition of dominated splitting, there are constants $C\geq 1$, $\lambda>0$ such that for every $t\geq 0$, one has that $$\dfrac{\left\lVert\psi^*_t|_{\widetilde{E_1}(x)} \right\rVert}{\left\lVert\psi^*_t|_{\widetilde{E_2}(x)}\right\rVert}\leq Ce^{-\lambda t}.$$ Conversely, suppose that $E_1(x)\nsubseteq\widetilde{E_1}(x)$. Then, there is a nonzero vector $u\in\widetilde{E_1}(x)\backslash E_1(x)$. On the other hand, since $\text{\rm dim}(E_1(x))=1\leq\text{\rm dim}(\widetilde{E_1}(x))$, $(\widetilde{E_1}(x))\nsubseteq E_1(x)$. Thus, there is a nonzero vector $v\in E_1(x)\backslash\widetilde{E_1}(x)$. Therefore, there is a vector decomposition $$u=u_1+u_2,\text{ where $u_1\in E_1(x)$, $u_2\in E_2(x)$ and $u_2\neq 0$},$$ of $u$ under the splitting $\mathcal{N}_x=E_1(x)\oplus E_2(x)$. Similarly, there is a vector decomposition $$v=v_1+v_2,\text{ where $v_1\in\widetilde{E_1}(x)$, $v_2\in\widetilde{E_2}(x)$ and $v_2\neq 0$},$$ of $v$ under the splitting $\mathcal{N}_x=\widetilde{E_1}(x)\oplus\widetilde{E_2}(x)$. Thus, 
\begin{equation*}
\begin{aligned}
   \dfrac{\lvert\psi^*_t(u)\rvert}{\lvert\psi^*_t(v)\rvert}&=\dfrac{\lvert\psi^*_t(u_1)+\psi^*_t(u_2)\rvert}{\lvert\psi^*_t(v)\rvert}\geq\dfrac{\lvert\psi^*_t(u_2)\lvert-\rvert\psi^*_t(u_1)\rvert}{\lvert\psi^*_t(v)\rvert}\geq\dfrac{e^{(\lambda_2-\epsilon)t}\lvert u_2\rvert-e^{(\lambda_1+\epsilon)t}\lvert u_1\rvert}{e^{(\lambda_1+\epsilon)t}\lvert v\rvert}\\&=\dfrac{e^{(\lambda_2-\lambda_1-2\epsilon)t}\lvert u_2\rvert-\lvert u_1\rvert}{\lvert v\rvert}\rightarrow+\infty\quad\text{(as $t\rightarrow+\infty$)},
\end{aligned}
\end{equation*}

   and 
\begin{equation*}
\begin{aligned}
   \dfrac{\lvert\psi^*_t(v)\rvert}{\lvert\psi^*_t(u)\rvert}&=\dfrac{\lvert\psi^*_t(v_1)+\psi^*_t(v_2)\rvert}{\lvert\psi^*_t(u)\rvert}\geq\dfrac{\lvert\psi^*_t(v_2)\lvert-\rvert\psi^*_t(v_1)\rvert}{\lvert\psi^*_t(u)\rvert}\geq\dfrac{\lvert\psi^*_t(v_2)\lvert-Ce^{-\lambda t}\lvert\psi^*_t(v_2)\lvert}{\lvert\psi^*_t(u)\rvert}\\&=\left(1-Ce^{-\lambda t}\right)\dfrac{\lvert\psi^*_t(v_2)\lvert}{\lvert\psi^*_t(u)\rvert}=\left(1-Ce^{-\lambda t}\right)C^{-1}e^{\lambda t}\dfrac{\lvert v_2\rvert}{\lvert u\rvert}\rightarrow+\infty\quad\text{(as $t\rightarrow+\infty$)}.
\end{aligned}
\end{equation*}
   This leads to a contradiction. Therefore, our assumption that $E_1(x)\nsubseteq\widetilde{E_1}(x)$ is not valid. It means that $E_1(x)=\widetilde{E_1}(x)$. Thus, the two splitting $\mathcal{N}_\Gamma=E_1\oplus E_2$ and $\mathcal{N}_\Gamma=\widetilde{E_1}\oplus\widetilde{E_2}$ coincide. It completes the proof of the claim.          
\end{proof}
	
   Let $C(M)$ be the set of all continuous functions on $M$, $\{f_i\}_{i=1}^{+\infty}$ be a countable dense subset of $C(M)$. For any measures $\mu,\nu\in\mathcal{M}(X)$, define $$d_{\mathcal{M}}(\mu,\nu)=\sum_{i=1}^{+\infty}\dfrac{\lvert\int_Mf_id\mu-\int_Mf_id\nu\rvert}{2^i\lVert f_i\rVert}.$$ Thus, $d_{\mathcal{M}}(\cdot,\cdot)$ is a metric which niduces the weak$^*$ topology in the set $\mathcal{M}(X)$. Given an ergodic hyperbolic invariant regular measure $\mu$, for every $\varepsilon>0$, we will prove that there exists a periodic measure $\nu$ such that $$d_{\mathcal{M}}(\mu,\nu)<\varepsilon.$$ Firstly, choose $n$ large enough such that
\begin{equation*}
   \sum_{i=n+1}^{+\infty}\dfrac{\lvert\int_Mf_id\mu-\int_Mf_id\nu\rvert}{2^i\lVert f_i\rVert}\lneq\sum_{i=n+1}^{+\infty}\dfrac{1}{2^{i-1}}<\dfrac{\varepsilon}{2}.\tag{\ddag}
\end{equation*}   
   According to the Birkhoff Ergodic Theorem, there is a full $\mu$-measure $\varphi$-invariant set $\Lambda_B$ such that for every $x\in\Lambda_B$ and every $f\in C(M)$, one has that $$\lim_{T\rightarrow+\infty}\dfrac{1}{T}\int^{T}_{0}f(\varphi_t(x))dt=\int_Mfd\mu.$$ Thus, there exists $T_1>1$ such that for every $i=1,2,\cdots,n$, every $x\in\Lambda_B$ and every $T>T_1$, one has that
\begin{equation*}
   \left\lvert\dfrac{1}{T}\int_{0}^{T}f_i(\varphi_t(x))dt-\int_Mf_id\mu\right\rvert<\dfrac{\varepsilon}{4n}\min_{1\leq i\leq n}\{2^i\lVert f_i\rVert\}.\tag{\ddag\ddag}
\end{equation*}      

   By the claim, the hyperbolic Oseledec splitting $\mathcal{N}_\Gamma=E_1\oplus E_2$ is a dominated splitting. Let $$\chi=\min\{\lvert\lambda_1\rvert,~\lvert\lambda_2\rvert\},$$ then by Lemma \ref{time}, for every $0<\epsilon\ll\chi$ and every $\delta\in(0,1)$, there is a positive number $L=L(\epsilon/2,\delta)$ such that
\begin{itemize}
   \item for every $T_0\geq L$, there exists a measurable set $\Gamma^{T_0}=\Gamma^{T_0}(\epsilon/2,\delta)\subset\Gamma$ with $\mu\left(\Gamma^{T_0}\right)\geq 1-\delta$;
   		
   \item there is a natural number $N=N(T_0)$ such that for every $n\geq N$ and any $x\in\Gamma^{T_0}$, we have that $$\prod_{i=0}^{n-1}\left\lVert\psi^*_{T_0}|_{E_1(\varphi_{iT_0}(x))}\right\rVert\leq e^{-(\chi-\epsilon/2)nT_0},~\prod_{i=0}^{n-1}m\left(\psi^*_{T_0}|_{E_2(\varphi_{iT_0}(x))}\right)\geq e^{(\chi-\epsilon/2)nT_0},$$ $$\dfrac{\left\lVert\psi^*_{T_0}|_{E_1(x)}\right\rVert}{m\left(\psi^*_{T_0}|_{E_2(x)}\right)}\leq e^{-T_0(\chi-\epsilon/2)}.$$     	  
\end{itemize}
   Fix an integer $T_0\geq\max\{T_1,~ L(\epsilon/2,\delta)\}$, let $\eta_0=(\chi-\epsilon/2)T_0$, for $k\in\mathbb{N}^+$, we can define the {\bf Pesin block } $\Lambda^{T_0}_{\eta_0}(k)$ as:
\begin{equation*}
\begin{aligned}
   \Lambda^{T_0}_{\eta_0}(k)=\Bigg\{x\in\Gamma:&\prod_{i=0}^{n-1}\left\lVert\psi^*_{T_0}|_{E_1(\varphi_{iT_0}(x))}\right\rVert\leq ke^{-n\eta_0},~\forall~ n\geq 1,\\ &\prod_{i=0}^{n-1}m\left(\psi^*_{T_0}|_{E_2(\varphi_{iT_0}(x))}\right)\geq ke^{n\eta_0},~\forall~ n\geq 1,~d(x,\text{Sing}(X))\geq\dfrac{1}{k}\Bigg\}.
\end{aligned}   
\end{equation*}
   According to \cite[Proposition 5.3]{WYZ}, $\Lambda^{T_0}_{\eta_0}(k)$ is a compact set and $$\mu\left(\Lambda^{T_0}_{\eta_0}(k)\right)\rightarrow\mu\left(\Gamma^{T_0}\right)\text{ as } k\rightarrow+\infty.$$ 

   Denote $$K=\max_{1\leq i\leq n}\left\{\lVert f_i\rVert\right\}.$$ Take $\gamma>0$ small enough, we may assume that $$(2 K+1)\gamma<\dfrac{\varepsilon}{4n}\cdot\min_{1\leq i\leq n}\{2^i\lVert f_i\rVert\}.$$ Recall that $\lvert X(x)\rvert\leq K_0$, for every $x\in M$. Since $f_i$ ($1\leq i\leq n$) is uniformly continuous on $M$, there exists $\xi\in[0,\gamma]$ small enough such that for every integer $i\in[1,n]$ and every $x,y\in M$ with $d(x,y)\leq\xi K_0$, we have that $$\left\lvert f_i(x)-f_i(y)\right\rvert<\gamma.$$ Fix $k$ large enough such that $\mu\left(\Lambda^{T_0}_{\eta_0}(k)\right)$ is sufficiently close to $1$. For this fixed $k$, there is $j=j(k)\in\mathbb{N}^+$ such that $k<e^{\frac{jT_0\epsilon}{2}}$. Thus, for every $x\in\Lambda^{T_0}_{\eta_0}(k)$, one has that $$\prod_{i=0}^{n-1}\left\lVert\psi^*_{jT_0}\lvert_{E_1(\varphi_{ijT_0}(x))}\right\rVert\leq e^{-(\chi-\epsilon)njT_0},~\prod_{i=0}^{n-1}m\left(\psi^*_{jT_0}\lvert_{E_2(\varphi_{ijT_0}(x))}\right)\geq e^{(\chi-\epsilon)njT_0},\quad\text{for }\forall~ n\geq 1.$$ Set $\eta=(\chi-\epsilon)jT_0$, $T=jT_0$, we consider the set   
\begin{equation*}
\begin{aligned}
   \Lambda^{T}_{\eta}(k)=\Bigg\{x\in\Gamma:&\prod_{i=0}^{n-1}\left\lVert\psi^*_{T}\lvert_{E_1(\varphi_{iT}(x))}\right\rVert\leq e^{-n\eta},~\forall~ n\geq 1,\\ &\prod_{i=0}^{n-1}m\left(\psi^*_{T}\lvert_{E_2(\varphi_{iT}(x))}\right)\geq e^{n\eta},~\forall~ n\geq 1,~d(x,\text{Sing}(X))\geq\dfrac{1}{k}\Bigg\}.
\end{aligned}   
\end{equation*}   
   Take a point $x_0\in\Lambda^T_\eta(k)$, by the continuity of $X(x)$, there is $r>0$ small enough such that for every $x\in B(x_0,r)$, we have that $$B(x_0,r)\subset B(x,\xi\lvert X(x)\rvert).$$ For $\varepsilon'=\min\{\varepsilon,~\xi\}>0$ and every $\alpha\in(0,1/k)$, we have a number $\delta_0>0$ by Theorem \ref{Shadow}. Since $\mu\left(B(x_0,r)\cap\Lambda^T_\eta(k)\cap\Lambda_B\right)>0$, by the Poincar\'{e} Recurrence Theorem, there is a point $y\in B(x_0,r)\cap\Lambda^T_\eta(k)\cap\Lambda_B$ and an integer $l$ such that $\varphi_{lT}(y)\in B(x_0,r)\cap\Lambda^T_\eta(k)\cap\Lambda_B$ with $d(y,\varphi_{lT}(y))<\delta_0$. Thus, we get an $(\eta,T)$-$\psi^*_t$-quasi hyperbolic orbit segment $\varphi_{[0,lT]}(y)$ satisfying
\begin{itemize}
   \item $d(y,\text{Sing}(X))>\alpha$ and $d(\varphi_{lT}(y),\text{Sing}(X))>\alpha$;
   		
   \item $y\in\Lambda^T_\eta(k)$, $\varphi_{lT}(y)\in\Lambda^T_\eta(k)$ and $d(y,\varphi_{lT}(y))<\delta_0$. 
\end{itemize}
   By Theorem \ref{Shadow}, there exists a $C^1$ strictly increasing function $\theta:[0,lT]\rightarrow\mathbb{R}$ and a periodic point $p\in M$ such that 
\begin{enumerate}
   \item[(1)] $\theta(0)=0$ and $1-\gamma\leq1-\xi<1-\varepsilon'<\theta'(t)<1+\varepsilon'<1+\xi\leq1+\gamma$, for every $t\in[0,lT]$; 
   
   \item[(2)] $p$ is periodic: $\varphi_{\theta(lT)}(p)=p$;
   
   \item[(3)] $d(\varphi_t(y),\varphi_{\theta(t)}(p))<\varepsilon'\lvert X(\varphi_t(x))\rvert<\xi\lvert X(\varphi_t(x))\rvert$, for every $t\in[0,lT]$.
\end{enumerate}   
   Therefore, for every $i=1,2,\cdots,n$, we have that 
\begin{equation*}
\begin{aligned}
   \left\lvert\dfrac{1}{\theta(lT)}\int_{0}^{\theta(lT)}f_i(\varphi_t(p))dt-\dfrac{1}{lT}\int_{0}^{\theta(lT)}f_i(\varphi_t(p))dt\right\rvert&\leq\left\lvert\dfrac{\theta(lT)}{lT}-1\right\rvert\cdot\dfrac{1}{\theta(lT)}\int_{0}^{\theta(lT)}f_i(\varphi_t(p))dt\\&\leq\gamma K.
\end{aligned}   
\end{equation*} 
   and    
\begin{equation*}
\begin{aligned}
   &\left\lvert\dfrac{1}{lT}\int_{0}^{\theta(lT)}f_i(\varphi_s(p))ds-\dfrac{1}{lT}\int_{0}^{lT}f_i(\varphi_t(y))dt\right\rvert=\dfrac{1}{lT}\left\lvert\int_{0}^{lT}f_i(\varphi_{\theta(t)}(p))d\theta(t)-\int_{0}^{lT}f_i(\varphi_t(y))dt\right\rvert\\&\leq\dfrac{1}{lT}\left\lvert\int_{0}^{lT}f_i(\varphi_{\theta(t)}(p))(\theta'(t)-1)dt\right\rvert+\dfrac{1}{lT}\left\lvert\int_{0}^{lT}\left[f_i(\varphi_{\theta(t)}(p))-f_i(\varphi_t(y))\right]dt\right\rvert\\&\leq\gamma K+\gamma.
\end{aligned}   
\end{equation*} 
   Thus, $$\left\lvert\dfrac{1}{\theta(lT)}\int_{0}^{\theta(lT)}f_i(\varphi_t(p))dt-\dfrac{1}{lT}\int_{0}^{lT}f_i(\varphi_t(y))dt\right\rvert<\dfrac{\varepsilon}{4n}\cdot\min_{1\leq i\leq n}\{2^i\lVert f_i\rVert\}.$$
   	
   Denote by $\mu_p$ the invariant measure on the periodic orbit ${\rm Orb}(p)$. From the equation $(\ddag)$ and $(\ddag\ddag)$, we have that    
\begin{equation*}
\begin{aligned}
   &d_{\mathcal{M}}(\mu,\mu_p)=\sum_{i=1}^{+\infty}\dfrac{\lvert\int_Mf_id\mu-\int_Mf_id\mu_p\rvert}{2^i\lVert f_i\rVert}=\sum_{i=1}^{n}\dfrac{\lvert\int_Mf_id\mu-\int_Mf_id\mu_p\rvert}{2^i\lVert f_i\rVert}+\sum_{i=n+1}^{+\infty}\dfrac{\lvert\int_Mf_id\mu-\int_Mf_id\mu_p\rvert}{2^i\lVert f_i\rVert}\\&\leq\sum_{i=1}^{n}\dfrac{\left\lvert\int_Mf_id\mu-\dfrac{1}{lT}\int_{0}^{lT}f_i(\varphi_t(y))dt+\dfrac{1}{lT}\int_{0}^{lT}f_i(\varphi_t(y))dt-\int_Mf_id\mu_p\right\rvert}{2^i\lVert f_i\rVert}+\sum_{i=n+1}^{+\infty}\dfrac{1}{2^{i-1}}\\&\leq\sum_{i=1}^{n}\dfrac{\left\lvert\int_Mf_id\mu-\dfrac{1}{lT}\int_{0}^{lT}f_i(\varphi_t(y))dt\right\rvert+\left\lvert\dfrac{1}{lT}\int_{0}^{lT}f_i(\varphi_t(y))dt-\dfrac{1}{\theta(lT)}\int_{0}^{\theta(lT)}f_i(\varphi_t(p))dt\right\rvert}{2^i\lVert f_i\rVert}+\dfrac{\varepsilon}{2}\\&<\sum_{i=1}^{n}\dfrac{\dfrac{\varepsilon}{4n}\cdot\min\limits_{1\leq i\leq n}\left\{2^i\lVert f_i\rVert\right\}+\dfrac{\varepsilon}{4n}\cdot\min\limits_{1\leq i\leq n}\left\{2^i\lVert f_i\rVert\right\}}{2^i\lVert f_i\rVert}+\dfrac{\varepsilon}{2}\leq\sum_{i=1}^{n}\dfrac{\varepsilon}{2n}+\dfrac{\varepsilon}{2}<\varepsilon.
\end{aligned}   
\end{equation*} 
   This completes the proof of Theorem \ref{ThmA}.         
\end{proof}

\section{ Approximation of Lyapunov exponents: proof of Theorem \ref{ThmB}}
   In this section, we will give the proof of Theorem \ref{ThmB}. Let $\varphi_t$ be the $C^1$ flow generated by a star vector field $X\in\mathfrak{X}^*(M^3)$ on three-dimensional manifolds, $\mu$ be an invariant hyperbolic ergodic measure of $\varphi_t$ which is not supported on singularities. Since every star vector field admits a dominated splitting $\mathcal{N}_{\text{supp}(\mu)\backslash\text{Sing}(X)}=E\oplus F$ w.r.t the scaled linear Poincar\'{e} flow $\psi^*_t$ such that $\text{\rm dim}(E)=\text{\rm Ind}(\mu)$, the Oseledec splitting $\mathcal{N}_\Gamma=E_1\oplus E_2$ of an invariant hyperbolic ergodic measure $\mu$ is a dominated splitting. According to the similar discussions on Lyapunov exponents as three different cases in the proof of Theorem \ref{ThmA}, we have that $$\lim_{t\to\pm\infty}\dfrac{1}{t}\log\parallel\psi^*_t(v) \parallel=\lambda_j,~\forall~v\in E_j,~v\neq 0,~j=1,2,$$ where $\lambda_1<0<\lambda_2$. To complete the proof of Theorem \ref{ThmB}, we present two propositions as Propositions \ref{TLLE} and \ref{TSLE}.   
   
\begin{Proposition}\label{TLLE}{\rm (The approximation of largest Lyapunov exponent)}
   Let $\mu$ be an ergodic hyperbolic invariant regular measure of the $C^1$ star vector field $X\in\mathfrak{X}^1(M)$ on a compact three dimensional smooth Riemannian manifold, $\lambda_1\leq\lambda_2$ be the Lyapunov exponents of $\mu$ with respect to the scaled linear Poincar\'{e} flow $\psi^*_t$. For every $\epsilon>0$, there is a hyperbolic periodic point $p$ of $\varphi_t$ such that $$\lvert\lambda_2-\lambda_2(p)\rvert<\epsilon,$$ where $\lambda_1(p)\leq\lambda_2(p)$ are the Lyapunov exponents of $\mu_p$ with respect to the scaled linear Poincar\'{e} flow $\psi^*_t$.     	
\end{Proposition}

\begin{proof}
   Let $$\chi=\min\{\lvert\lambda_1\rvert,~\lambda_2\}.$$ Since the Oseledec splitting $\mathcal{N}_\Gamma=E_1\oplus E_2$ is a dominated splitting, by Lemma \ref{time}, for every $0<\epsilon/4\ll\chi$ and every $\delta_1\in(0,1/9)$, there is a positive number $L=L(\epsilon/4,\delta_1)$ such that
\begin{itemize}
   \item for every $T\geq L$, there exists a measurable set $\Gamma^T=\Gamma^T(\epsilon/4,\delta_1)\subset\Gamma$ with $\mu\left(\Gamma^T\right)\geq 1-\delta_1$;
    	
   \item there is a natural number $N=N(T)$ such that for every $n\geq N$ and every $x\in\Gamma^T$, we have that $$\prod_{i=0}^{n-1}\left\lVert\psi^*_T|_{E_1(\varphi_{iT}(x))}\right\rVert\leq e^{-(\chi-\epsilon/4)nT},~\prod_{i=0}^{n-1}m\left(\psi^*_T|_{E_2(\varphi_{iT}(x))}\right)\geq e^{(\chi-\epsilon/4)nT}.$$ $$\dfrac{\left\lVert\psi^*_T|_{E_1}\right\rVert}{m\left(\psi^*_T\|_{E_2}\right)}\leq e^{-T(\chi-\epsilon/4)}.$$     	  
\end{itemize}
   Fix an integer $T_0\geq L(\epsilon/4,\delta_1)$, let $\eta_0=(\chi-\epsilon/4)T_0$, for $k\in\mathbb{N}^+$, we define the {\bf Pesin block } $\Lambda^{T_0}_{\eta_0}(k)$ as:
\begin{equation*}
\begin{aligned}
   \Lambda^{T_0}_{\eta_0}(k)=\Bigg\{x\in\Gamma:&\prod_{i=0}^{n-1}\left\lVert\psi^*_{T_0}|_{E_1(\varphi_{iT_0}(x))}\right\rVert\leq ke^{-n\eta_0},~\forall~ n\geq 1,\\ &\prod_{i=0}^{n-1}m\left(\psi^*_{T_0}|_{E_2(\varphi_{iT_0}(x))}\right)\geq ke^{n\eta_0},~\forall~ n\geq 1,~d(x,\text{Sing}(X))\geq\dfrac{1}{k}\Bigg\}.
\end{aligned}   
\end{equation*}
   According to \cite[Proposition 5.3]{WYZ}, $\Lambda^{T_0}_{\eta_0}(k)$ is a compact set and $$\mu\left(\Lambda^{T_0}_{\eta_0}(k)\right)\rightarrow\mu\left(\Gamma^{T_0}\right)\text{ as } k\rightarrow+\infty.$$ Therefore, one can fix $k$ large enough, such that $\mu\left(\Lambda^{T_0}_{\eta_0}(k)\right)>1-2\delta_1>0$. For this fixed $k$, there is $j=j(k)\in\mathbb{N}^+$ such that $k<e^{\frac{jT_0\epsilon}{4}}$. Thus, for every $x\in\Lambda^{T_0}_{\eta_0}(k)$, one has that $$\prod_{i=0}^{n-1}\left\lVert\psi^*_{jT_0}|_{E_1(\varphi_{ijT_0}(x))}\right\rVert\leq e^{-(\chi-\epsilon/2)njT_0},~\prod_{i=0}^{n-1}m\left(\psi^*_{jT_0}|_{E_2(\varphi_{ijT_0}(x))}\right)\geq e^{(\chi-\epsilon/2)njT_0},\quad\text{for }\forall~ n\geq 1.$$ Set $\eta=(\chi-\epsilon/2)jT_0$, $T=jT_0$, we consider the set
\begin{equation*}
\begin{aligned}
   \Lambda^{T}_{\eta}(k)=\Bigg\{x\in\Gamma:&\prod_{i=0}^{n-1}\left\lVert\psi^*_{T}\lvert_{E_1(\varphi_{iT}(x))}\right\rVert\leq e^{-n\eta},~\forall~ n\geq 1,\\ &\prod_{i=0}^{n-1}m\left(\psi^*_{T}\lvert_{E_2(\varphi_{iT}(x))}\right)\geq e^{n\eta},~\forall~ n\geq 1,~d(x,\text{Sing}(X))\geq\dfrac{1}{k}\Bigg\}.
\end{aligned}   
\end{equation*}       
   
   Now, we prove the proposition by three steps.	
	
   {\bf Step $1$}: In this step, we will prove that there exists a periodic point $p$ such that $$\lambda_2(p)\leq\lambda_2+\epsilon.$$ 

   For $\epsilon/4>0$, there is a number $\zeta>0$, such that $$\left\lvert\log(1+\zeta)\right\rvert<\dfrac{\epsilon}{4},\text{ and }\left\lvert\log(1-\zeta)\right\rvert<\dfrac{\epsilon}{4}.$$ For every $x_1,~x_2\in\Lambda^{T}_{\eta}\left(k\right)$, denote by $${\rm I}(x_1,x_2)=\left(\psi^*_T|_{\mathcal{N}(x_1)}\right)^{-1}\circ\left(\psi^*_T|_{\mathcal{N}(x_2)}\right).$$ Since the Oseledec splitting $\mathcal{N}_\Gamma=E_1\oplus E_2$ is a dominated splitting, it is easy to see that ${\rm I}(x_1,x_2)$ approach to identity if $x_1$ and $\varphi_T(x_2)$ close enough. Let $\sigma_1>0$ satisfy that $$d(x_1,x_2)\leq\sigma_1\Longrightarrow 1-\zeta\leq\lVert {\rm I}(x_1,x_2)\rVert\leq1+\zeta\Longrightarrow\lvert\log\lVert {\rm I}(x_1,x_2)\rVert\rvert<\dfrac{\epsilon}{4}.$$ Since $\Lambda^{T}_{\eta}$ is compact, $\sigma_1$ always exists. Fix $t_0$, the norm $$\lVert\psi^*_{t_0}\rVert=\sup\left\{\lvert\psi^*_{t_0}(v)\rvert:~v\in\mathcal{N},~\lvert v\rvert=1\right\}$$ is uniformly upper bounded on $\mathcal{N}$, and the norm $$m\left(\psi^*_{t_0}\right)=\inf\left\{\lvert\psi^*_{t_0}(v)\rvert:~v\in\mathcal{N},~\lvert v\rvert=1\right\}$$ is uniformly bounded from $0$ on $\mathcal{N}$. Without loss of generality, we may assume that $$J_1<m\left(\psi^*_{t_0}\right)<\lVert\psi^*_{t_0}\rVert<J_2.$$ Since $\varphi_t$ is a $C^1$ flow, for $\zeta/J_2>0$, there is $\sigma_2>0$ such that for every $s\in(t_0-\sigma_2,~t_0+\sigma_2)$, one has that $\psi^*_{t_0}|_{\mathcal{N}_{\varphi_s(x)}}$ is sufficiently close to $\psi^*_{t_0}|_{\mathcal{N}_{\varphi_{t_0}(x)}}$ and $$\lVert\psi^*_{t_0}|_{\mathcal{N}_x}\rVert-\zeta/J_2\leq\lVert\psi^*_s|_{\mathcal{N}_x}\rVert\leq\lVert\psi^*_{t_0}|_{\mathcal{N}_x}\rVert+\zeta/J_1,\quad\text{for every }x\in\Lambda^{T}_{\eta}(k).$$ Thus, we have that
\begin{equation*}
\begin{aligned}
   1-\zeta&\leq\left\lVert\left(\psi^*_{t_0}|_{\mathcal{N}_{\varphi_s(x)}}\right)^{-1}\right\rVert\left(\lVert\psi^*_{t_0}|_{\mathcal{N}_x}\rVert-\zeta/J_2\right)\leq\left\lVert\left(\psi^*_{t_0}|_{\mathcal{N}_{\varphi_s(x)}}\right)^{-1}\circ\psi^*_s|_{\mathcal{N}_x}\right\rVert\\&\leq\left\lVert\left(\psi^*_{t_0}|_{\mathcal{N}_{\varphi_s(x)}}\right)^{-1}\right\rVert\left(\lVert\psi^*_{t_0}|_{\mathcal{N}_x}\rVert+\zeta/J\right)\leq 1+\zeta
\end{aligned}
\end{equation*}   
   Then, $$\left\lvert\log\left\lVert\left(\psi^*_{t_0}|_{\mathcal{N}_{\varphi_s(x)}}\right)^{-1}\circ\psi^*_s|_{\mathcal{N}_x}\right\rVert\right\rvert<\dfrac{\epsilon}{4}.$$
   
   Take $y\in\Lambda^{T}_{\eta}(k)$, by Poincar\'{e} Recurrence Theorem, there is an increasing integer sequence $\{l_n\}$ such that $$d(y,\varphi_{l_nT}(y))\longrightarrow 0.$$ For $0<\varepsilon<\min\left\{\dfrac{\epsilon}{4},\dfrac{\sigma_1}{K_0},\sigma_2\right\}$ and $\alpha\in(0,1/k)$, we get a constant $\delta_2=\delta_2(\varepsilon,\alpha)$ by Theorem \ref{Shadow}. For $l_n$ large enough, we can obtain a $(\eta,T)$-$\psi^*_t$-quasi hyperbolic orbit segment $\varphi_{[0,l_nT]}(y)$ which satisfies that
\begin{itemize}
   \item $d(y,\text{Sing}(X))>\alpha$ and $d(\varphi_{l_nT}(y),\text{Sing}(X))>\alpha$;
   	
   \item $y\in\Lambda^{T}_{\eta}(k)$, $\varphi_{l_nT}(y)\in\Lambda^{T}_{\eta}(k)$ and $d(y,\varphi_{l_nT}(y))<\delta_2$. 
\end{itemize}
   According to Theorem \ref{Shadow}, there is a $C^1$ strictly increasing function $\theta:[0,l_nT]\rightarrow\mathbb{R}$ and a periodic point $p\in M$ such that
\begin{itemize}
   \item[(1)] $\theta(0)=0$ and $1-\varepsilon<\theta'(t)<1+\varepsilon$, for every $t\in[0,l_nT]$;
   	
   \item[(2)] $p$ is periodic: $\varphi_{\theta(l_nT)}(p)=p$;
   	
   \item[(3)] $d(\varphi_t(y),\varphi_{\theta(t)}(p))<\varepsilon\lvert X(\varphi_t(y))\rvert$, for every $t\in[0,l_nT]$.
\end{itemize} 
   Moreover, by the Proposition 4.4 in \cite{WYZ}, there is constant  $N=N(\eta,T)$ such that $$\lvert\theta(iT)-iT\rvert\leq Nd(y,\varphi_{l_nT}(y)),\quad\text{for }i=1,2,\cdots,l_n.$$ Take $l_n$ large enough, such that $$d(y,\varphi_{l_nT}(y))<\dfrac{\sigma_2}{2N}.$$ We will write $l_n$ simply $l$ when no confusion can arise. Therefore, we have that $$\dfrac{1}{lT}\sum_{i=0}^{l-1}\log\left\lVert\psi^*_{T}\lvert_{\mathcal{N}_{\varphi_{iT}(y)}}\right\rVert\leq\lambda_2+\dfrac{\epsilon}{2}.$$ To calculate the Lyapunov exponents of the periodic atomic measure $\mu_p$, we give the following claim.  
   
\begin{claim}
   For any $i=0,1,\cdots,l-1$, let $$T_i(y,p)=\left(\psi^*_T|_{\mathcal{N}_{\varphi_{(i+1)T}(y)}}\right)^{-1}\circ\psi^*_{\theta((i+1)T)-\theta(iT)}|_{\mathcal{N}_{\varphi_{\theta(iT)}(p)}},$$ then we have that $$\left\lvert\log\left\lVert T_i(y,p)\right\rVert\right\rvert<\dfrac{\epsilon}{2}.$$  	
\end{claim}

\begin{proof}[proof of the claim]
   For every $i=0,1,\cdots,l-1$, we have that $$\psi^*_T|_{\mathcal{N}_{\varphi_{\theta((i+1)T)-T}(p)}}\circ\left(\psi^*_T|_{\mathcal{N}_{\varphi_{\theta((i+1)T)}(p)}}\right)^{-1}$$ is identity. Thus, we consider $$\left(\psi^*_T|_{\mathcal{N}_{\varphi_{(i+1)T}(y)}}\right)^{-1}\circ\psi^*_T|_{\mathcal{N}_{\varphi_{\theta((i+1)T)-T}(p)}}\circ\left(\psi^*_T|_{\mathcal{N}_{\varphi_{\theta((i+1)T)}(p)}}\right)^{-1}\circ\psi^*_{\theta((i+1)T)-\theta(iT)}|_{\mathcal{N}_{\varphi_{\theta(iT)}(p)}}.$$ Since $$\lvert\theta(iT)-iT\rvert\leq Nd(y,\varphi_{lT}(y)),\quad\text{for }i=0,1,2,\cdots,l,$$ we have that
\begin{equation*}
\begin{aligned}
   &\left\lvert\theta((i+1)T)-T-\theta(iT)\right\rvert= \left\lvert\theta((i+1)T)-\theta(iT)-((i+1)T-T)\right\rvert\\=&\left\lvert\theta((i+1)T)-(i+1)T-(\theta(iT)-T)\right\rvert\leq\left\lvert\theta((i+1)T)-(i+1)T\right\rvert+\left\lvert(\theta(iT)-T)\right\rvert\\\leq&2Nd(y,\varphi_{lT}(y))<\sigma_2.
\end{aligned}  
\end{equation*}     
   Therefore, $\psi^*_T|_{\mathcal{N}_{\varphi_{\theta(iT)}(p)}}$ is sufficiently close to $\psi^*_{\theta((i+1)T)-\theta(iT)}|_{\mathcal{N}_{\varphi_{\theta(iT)}(p)}}$. Thus, $$\left(\psi^*_T|_{\mathcal{N}_{\varphi_{\theta((i+1)T)}(p)}}\right)^{-1}\circ\psi^*_{\theta((i+1)T)-\theta(iT)}|_{\mathcal{N}_{\varphi_{\theta(iT)}(p)}}$$ is sufficiently close to identity $$\left(\psi^*_{\theta((i+1)T)-\theta(iT)}|_{\mathcal{N}_{\varphi_{\theta((i+1)T)}(p)}}\right)^{-1}\circ\psi^*_{\theta((i+1)T)-\theta(iT)}|_{\mathcal{N}_{\varphi_{\theta(iT)}(p)}}.$$ On the other hand, since $$d(\varphi_t(y),\varphi_{\theta(t)}(p))<\varepsilon\lvert X(\varphi_t(y))\rvert,\quad\text{for every }t\in[0,lT],$$ we have that $$d(\varphi_{iT}(y),\varphi_{\theta(iT)}(p))<\varepsilon\lvert X(\varphi_{iT}(y))\rvert<\sigma_1,\quad\text{for every }i=0,1,2,\cdots,l.$$ Consequently, $$\left\lvert\log\left\lVert\left(\psi^*_T|_{\mathcal{N}_{\varphi_{(i+1)T}(y)}}\right)^{-1}\circ\psi^*_T|_{\mathcal{N}_{\varphi_{\theta((i+1)T)-T}(p)}}\right\rVert\right\rvert<\dfrac{\epsilon}{4}.$$ Then, $$\left\lvert\log\left\lVert T_i(y,p)\right\rVert\right\rvert<\dfrac{\epsilon}{2}.$$       	
\end{proof}
   Since$$\psi^*_{\theta(lT)}|_{\mathcal{N}_p}=\psi^*_T|_{\mathcal{N}_{\varphi_{(l-1)T}(y)}}\circ T_{l-1}(y,p)\circ\cdots\circ\psi^*_T|_{\mathcal{N}_{\varphi_T(y)}}\circ T_1(y,p)\circ\psi^*_T|_{\mathcal{N}_y}\circ T_0(y,p),$$ we have that 
\begin{equation*}
\begin{aligned}
   \lambda_2(p)&=\lim_{j\rightarrow+\infty}\dfrac{1}{j\theta(lT)}\log\left\lVert\psi^*_{\theta(lT)}|_{\mathcal{N}_p}\right\rVert=\lim_{j\rightarrow+\infty}\dfrac{1}{j\theta(lT)}\log\left\lVert \left(\psi^*_{\theta(lT)}|_{\mathcal{N}_p}\right)^j\right\rVert\\&\leq\dfrac{1}{\theta(lT)}\log\left\lVert\psi^*_{\theta(lT)}|_{\mathcal{N}_p}\right\rVert\leq\dfrac{1}{\theta(lT)}\left(\sum_{i=0}^{l-1}\log\left\lVert\psi^*_T|_{\mathcal{N}_{\varphi_{iT}(y)}}\right\rVert+\sum_{i=0}^{l-1}\log\left\lVert T_i(y,p)\right\rVert\right).
\end{aligned}
\end{equation*}
   Since $1-\varepsilon<\theta'(t)<1+\varepsilon$, for every $t\in[0,lT]$, we have that $$(1-\varepsilon)lT<\lvert\theta(lT)\rvert<(1+\varepsilon)lT.$$ Therefore, $$\lambda_2(p)\leq\dfrac{lT(1+\varepsilon)}{\theta(lT)}\cdot\dfrac{1}{lT}\left(\sum_{i=0}^{l-1}\log\left\lVert\psi^*_T|_{\mathcal{N}_{\varphi_{iT}(y)}}\right\rVert+\sum_{i=0}^{l-1}\log\left\lVert T_i(y,p)\right\rVert\right)\leq\lambda_2+\epsilon.$$

   {\bf Step $2$}: In this step, we will prove that there exists a periodic point $p$ such that $$\lambda_2(p)\geq\lambda_2-\epsilon.$$ 
   
   For estimating the lower bound approximation of Lyapunov exponents, we give a following claim on invertible linear maps.
\begin{claim}
   Given two invertible linear maps $A,~B:\mathbb{R}^d\longrightarrow\mathbb{R}^d$. For $\epsilon_1>0$, there is $\sigma_3>0$ such that if $\lVert B-I\rVert\leq\sigma_3$, then 
\begin{equation*}
   \lVert ABv\rVert\geq e^{-\epsilon_1}\lVert Av\rVert,\quad \forall~v\in\mathbb{R}^n. \tag{\S}
\end{equation*}	
\end{claim} 

\begin{proof}
   We only need to prove the inequality $(\S)$ for $v\in\mathbb{R}^n$ with $\lVert v\rVert=1$. Now, we always assume that $\lVert v\rVert=1$. For any $\epsilon_1>0$, let $\epsilon_2>0$ satisfying that $1-\epsilon_2>e^{-\epsilon_1}$ and $$\sigma_3=\dfrac{m(A)}{\lVert A\rVert}\epsilon_2.$$ Therefore, $$\lVert A\rVert\cdot\lVert B-I\rVert\leq\epsilon_2m(A).$$ Consequently, if $\lVert B-I\rVert\leq\sigma_3$, we have that
\begin{equation*}
\begin{aligned}
   \lVert ABv\rVert&=\lVert Av+A(B-I)v\rVert\geq\lVert Av\rVert-\lVert A\rVert\cdot\lVert B-I\rVert\cdot\lVert v\rVert=\lVert Av\rVert-\lVert A\rVert\cdot\lVert B-I\rVert\\&\geq\lVert Av\rVert-\epsilon_2m(A)\geq(1-\epsilon_2)\lVert Av\rVert\geq e^{-\epsilon_1}\lVert Av\rVert.
\end{aligned}
\end{equation*}      
\end{proof}
     
   For given $x\in M$ and an enough small neighborhood $U(x)$ of $x$, by the Lyapunov chart map ${\rm Exp}_x$, we can locally trivialize the tangent bundle $T_{U(x)}M=U(x)\times\mathbb{R}^d$ of $U(x)$. Thus, for every $y\in U(x)$, we can translate the vectors of $T_xM$ to $T_yM$. So, we can translate the vectors of $\mathcal{N}_x$ to $\mathcal{N}_y$. According to Lemma \ref{Dlgx}, there is a $\beta^*_0>0$ such that for every $x\in\Lambda^{T}_{\eta}(k)$, every point $y\in U(x)=B(x,\beta^*_0\lvert X(x)\rvert)$ is not a singularity. Thus, for every $v\in\mathcal{N}_y$ and $\rho>1$, we can define the $\rho$-cone in $\mathcal{N}_y$: $$\mathcal{C}_\rho(y)=\{v\in\mathcal{N}_y:~\lVert v_2\rVert\geq\rho\lVert v_1\rVert\},$$ where $v=v_1+v_2$, $v_j\in E_j(x)$, $E_j(x)$ is the Oseledec subspace with respect to the Lyapunov exponent $\lambda_j$ for $j=1,2$. Namely, translating the splitting $E_1(x)\oplus E_2(x)$ of $\mathcal{N}_x$ to $\mathcal{N}_y$, and for $v\in\mathcal{N}_y$, $v=v_1+v_2$ is the expression corresponding to the splitting. For compution, we can define a equivalent norm $\lVert\cdot\rVert'$ on $\mathcal{N}_y$: $$\lVert v\rVert'=\max\{\lVert v_1\rVert,\lVert v_2\rVert\rVert\}.$$ Since the Oseledec splitting  $$\mathcal{N}_\Gamma=E_1\oplus E_2$$ is a dominated splitting, for $\psi^*_T$, there exist two positive number $\gamma\in(0,1)$ and $\rho>1$ such that $\gamma\rho>1$ and $$\psi^*_T\mathcal{C}_{\gamma\rho}(x)\subset\mathcal{C}_{\rho}(\varphi_{T}(x)),\quad\forall~x\in\Lambda^{T}_{\eta}(k).$$ Furthermore, we can choose $0<\sigma_4=\sigma_4(\gamma,\rho)$ satisfying that $$\lVert T-I\rVert\leq\sigma_4\Longrightarrow T\mathcal{C}_{\rho}\subset\mathcal{C}_{\gamma\rho}.$$
   
   Take $y\in\Lambda^{T}_{\eta}(k)$, by Poincar\'{e} Recurrence Theorem, there is an increasing integer sequence $\{l_n\}$ such that $$d(y,\varphi_{l_nT}(y))\longrightarrow 0.$$ Let $$\sigma_3=\dfrac{m\left(\psi^*_T\right)}{\left\lVert\psi^*_T\right\rVert}\left(1-e^{-\epsilon/4}\right).$$ For $0<\varepsilon<\min\left\{\dfrac{\epsilon}{4},\dfrac{\sigma_1}{K_0},\sigma_2,\sigma_3,\log(1+\sigma_4)\right\}$ and $\alpha\in(0,1/k)$, we get a constant $\delta_2=\delta_2(\varepsilon,\alpha)$ by the Theorem \ref{Shadow}. For $l_n$ large enough, we obtain a $(\eta,T)$-$\psi^*_t$-quasi hyperbolic orbit segment $\varphi_{[0,l_nT]}(y)$ satisfying that
\begin{itemize}
   \item $d(y,\text{Sing}(X))>\alpha$ and $d(\varphi_{l_nT}(y),\text{Sing}(X))>\alpha$;
  	
   \item $y\in\Lambda^{T}_{\eta}(k)$, $\varphi_{l_nT}(y)\in\Lambda^{T}_{\eta}(k)$ and $d(y,\varphi_{l_nT}(y))<\delta_2$. 
\end{itemize}
   According to Theorem \ref{Shadow}, there is a $C^1$ strictly increasing function $\theta:[0,l_nT]\rightarrow\mathbb{R}$ and a periodic point $p\in M$ such that
\begin{itemize}
   \item[(1)] $\theta(0)=0$ and $1-\varepsilon<\theta'(t)<1+\varepsilon$, for every $t\in[0,l_nT]$;
  	
   \item[(2)] $p$ is periodic: $\varphi_{\theta(l_nT)}(p)=p$;
  	
   \item[(3)] $d(\varphi_t(y),\varphi_{\theta(t)}(p))<\varepsilon\lvert X(\varphi_t(y))\rvert$, for every $t\in[0,l_nT]$.
\end{itemize}    
   Moreover, by the Proposition 4.4 in \cite{WYZ}, there is constant  $N=N(\eta,T)$ such that $$\lvert\theta(iT)-iT\rvert\leq Nd(y,\varphi_{l_nT}(y)),\quad\text{for }i=1,2,\cdots,l_n.$$ Take $l_n$ large enough, such that $$d(y,\varphi_{l_nT}(y))<\dfrac{\sigma_2}{2N}.$$ We will write $l_n$ simply $l$ when no confusion can arise. By the similar discussions as the claim in Step $1$, we have that $$\lVert T_i(y,p)-I\rVert\leq\sigma_4,\text{ for every }i=0,1,\cdots,l-1$$ and $$\lVert T_i(y,p) v\rVert'\geq e^{-\frac{\epsilon}{2}}\lVert v\rVert',\quad\forall~ v\in\mathcal{N}_{\varphi_{\theta(iT)}(p)}\big\backslash\{0\},\text{ for every }i=0,1,\cdots,l-1.$$ Therefore, for each $i=0,1,\cdots,l-1$, we have that
\begin{equation*}
   \psi^*_{T}|_{\mathcal{N}_{\varphi_{iT}(y)}}\circ T_i(y,p)\mathcal{C}_{\rho}\left(\varphi_{\theta(iT)}(p)\right)\subset\mathcal{C}_{\rho}\left(\varphi_{\theta((i+1)T)}(p)\right).\tag{$\vartriangle$}
\end{equation*}   
  
   For every $v=(v_1,v_2)\in\mathcal{C}_\rho(p)$ and $i=0,1,\cdots,l-1$, denote by $$\bar{v}^0=v,\quad \bar{v}^i=T_i(y,p)\circ\psi^*_{T}|_{\mathcal{N}_{\varphi_{(i-1)T}(y)}}\circ T_{i-1}(y,p)\circ\cdots\circ\psi^*_{T}|_{\mathcal{N}_y}\circ T_0(y,p)(v)$$ and $$v^{i+1}=\psi^*_{T}|_{\mathcal{N}_{\varphi_{iT}(y)}}\circ T_i(x,p)\circ\psi^*_{T}|_{\mathcal{N}_{\varphi_{(i-1)T}(y)}}\circ T_{i-1}(x,p)\circ\cdots\circ\psi^*_{T}|_{\mathcal{N}_y}\circ T_0(y,p)(v).$$ Corresponding the Oseledec splitting $$\mathcal{N}_\Gamma=E_1\oplus E_2,$$ denote by $$\bar{v}^i=\bar{v}^i_1+\bar{v}^i_2,\quad v^i=v^i_1+v^i_2.$$ Then, for each $i=0,1,\cdots,l-1$, we have that
\begin{equation}
\begin{aligned}
   \left\lVert v^{i+1}\right\rVert'=\left\lVert v^{i+1}_2\right\rVert&=\left\lVert\psi^*_T(\bar{v}^i_2)\right\rVert\geq m\left(\psi^*_T|_{E_2(\varphi_{iT}(y))}\right)\left\lVert\bar{v}^i_k\right\rVert=m\left(\psi^*_{T}|_{E_2(\varphi_{iT}(y))}\right)\left\lVert\bar{v}^i\right\rVert'\\&=m\left(\psi^*_{T}|_{E_2(\varphi_{iT}(y))}\right)\left\lVert T_i(x,p)v^i\right\rVert'\geq e^{-\frac{\epsilon}{2}}m\left(\psi^*_{T}|_{E_2(\varphi_{iT}(y))}\right)\left\lVert v^i\right\rVert'\\&\geq e^{-\frac{\epsilon}{2}(i+1)}\prod^{i}_{j=0}m\left(\psi^*_{T}|_{E_2(\varphi_{jT}(y))}\right)\lVert v\rVert'.
\end{aligned}
\end{equation}   
   Since $\mathcal{C}_\rho(p)$ contains the subspace $E_2(p)$, formulas $(\Delta)$, $(1)$ and Lemma \ref{time} indicate that $$\lambda_2(p)\geq\lambda_2-\epsilon.$$ 
   
   {\bf Step $3$}: Take $y\in\Lambda^{T}_{\eta}(k)$, by Poincar\'{e} Recurrence Theorem, there is an increasing integer sequence $\{l_n\}$ such that $$d(y,\varphi_{l_nT}(y))\longrightarrow 0.$$ For the numbers $\sigma_1,~\sigma_2$ given in {\bf Step $1$}, and $\sigma_3,~\sigma_4$ given in {\bf Step $2$}, for $$0<\varepsilon<\min\left\{\dfrac{\epsilon}{4},\dfrac{\sigma_1}{K_0},\sigma_2,\sigma_3,\log(1+\sigma_4)\right\}$$ and $\alpha\in(0,1/k)$, we get a constant $\delta_2=\delta_2(\varepsilon,\alpha)$ by the Theorem \ref{Shadow}. For $l_n$ large enough, we obtain a $(\eta,T)$-$\psi^*_t$-quasi hyperbolic orbit segment $\varphi_{[0,l_nT]}(y)$ satisfying that
\begin{itemize}
   \item $d(y,\text{Sing}(X))>\alpha$ and $d(\varphi_{l_nT}(y),\text{Sing}(X))>\alpha$;
  	
   \item $y\in\Lambda^{T}_{\eta}(k)$, $\varphi_{l_nT}(y)\in\Lambda^{T}_{\eta}(k)$ and $d(y,\varphi_{l_nT}(y))<\delta_2$. 
\end{itemize}
   According to Theorem \ref{Shadow}, there is a $C^1$ strictly increasing function $\theta:[0,l_nT]\rightarrow\mathbb{R}$ and a periodic point $p\in M$ such that
\begin{itemize}
   \item[(1)] $\theta(0)=0$ and $1-\varepsilon<\theta'(t)<1+\varepsilon$, for every $t\in[0,l_nT]$;
  	
   \item[(2)] $p$ is periodic: $\varphi_{\theta(l_nT)}(p)=p$;
  	
   \item[(3)] $d(\varphi_t(y),\varphi_{\theta(t)}(p))<\varepsilon\lvert X(\varphi_t(y))\rvert$, for every $t\in[0,l_nT]$.
\end{itemize}    
   Moreover, by the Proposition 4.4 in \cite{WYZ}, there is constant  $N=N(\eta,T)$ such that $$\lvert\theta(iT)-iT\rvert\leq Nd(y,\varphi_{l_nT}(y)),\quad\text{for }i=1,2,\cdots,l_n.$$ Take $l_n$ large enough, such that $$d(y,\varphi_{l_nT}(y))<\dfrac{\sigma_2}{2N}.$$ We will write $l_n$ simply $l$ when no confusion can arise. Thus, we get a periodic orbit of $p$ with period $\theta(lT)$. Then, we have an invariant measure $$\mu_p=\dfrac{1}{\theta(lT)}\int_{0}^{\theta(lT)}\delta_{\varphi_t(p)}dt$$ supported on the periodic orbit of $p$. Let $$\lambda_1(p)\leq\lambda_2(p)$$ be the Lyapunov exponents of $\mu_p$. By Step $1$ and Step $2$, we have that $$\lvert\lambda_2(p)-\lambda_2\rvert<\epsilon.$$ This proves the proposition.
\end{proof}

   Now, we provide some notations and basic facts for the vector field $-X$. Given a vector field $X\in\mathfrak{X}^1(M^d)$, denote by $-X$ the vector field which satisfies that $-X(x)$ and $X(x)$ have opposite directions, $\lvert-X(x)\rvert=\lvert X(x)\rvert$, for every $x\in M$. Therefore, the vector field $-X$ can also generate a $C^1$ flow which is denoted by $\overline{\varphi_t}$. Correspondingly, we can obtain the tangent flow $\overline{\Phi_t}$, the linear Poincar\'{e} flow $\overline{\psi_t}$, the extended linear Poincar\'{e} flow $\overline{\widetilde{\psi}_t}$ and scaled linear Poincar\'{e} $\overline{\psi^*_t}$.
   
   According to the relationship between $\overline{\varphi_t}$ and $\varphi_t$, every ergodic invariant measure $\mu$ of $\varphi_t$ is also an ergodic invariant measure of $\overline{\varphi}_t$. Therefore, by the Oseledec Theorem \cite{OV}, for $\mu$-almost every $x$, there are a positive integer $k\in[1,d]$, real numbers $\overline{\chi_1}<\cdots<\overline{\chi_{k}}$ and a measurable $\overline{\Phi_t}$-invariant splitting $$T_xM=\overline{E_1}(x)\oplus\overline{E_2}(x)\oplus\cdots\oplus \overline{E_k}(x)$$ such that $$\lim_{t\to\pm\infty}\dfrac{1}{ t}\log\parallel\overline{\Phi_t}(v)\parallel=\overline{\chi_i},~\forall~v\in \overline{E_i}(x),~v\neq 0,~i=1,2,\cdots,k.$$ The numbers $\overline{\chi_1},\cdots,\overline{\chi_{k}}$ are called \emph{\bf Lyapunov exponents} at point $x$ of $\overline{\Phi_t}$ with respect to $\mu$. Denote by $$\overline{d_i}=\text{dim}(\overline{E_i}(x)),\quad i=1,2,\cdots,k$$ the multiplicities of those Lyapunov exponents, and the vector formed by these numbers (counted with multiplicity, endowed with the increasing order) is called \emph{\bf Lyapunov vector} at point $x$ of $\overline{\Phi_t}$ with respect to $\mu$. Regardless of the multiplicity, we can rewrite the Lyapunov exponents of $\mu$ by $\overline{\lambda_1}\leq\overline{\lambda_2}\leq\cdots\leq\overline{\lambda_d}$. Namely, $$\overline{\lambda_j}=\overline{\chi_i},\quad\text{ for each } \overline{d_1}+\overline{d_2}+\cdots+\overline{d_{i-1}}<j\leq \overline{d_1}+\overline{d_2}+\cdots+\overline{d_i}.$$ 
   
   For an ergodic invariant measure $\mu$ which is not concentrated on ${\rm Sing}(X)$, according to the definition of linear Poincar\'e flow $\overline{\psi_t}:\mathcal{N}\rightarrow\mathcal{N}$ and the Oseledec Theorem \cite{OV}, for $\mu$-almost every $x$, there are a positive integer $k\in[1,d-1]$, real numbers $\overline{\chi_1}<\overline{\chi_2}<\cdots<\overline{\chi_k}$ and a measurable $\overline{\psi_t}$-invariant splitting (for simplicity, we omit the base point) $$\mathcal{N}=\overline{E_1}\oplus \overline{E_2}\oplus\cdots\oplus\overline{E_k}$$ on normal bundle such that $$\lim_{t\to\pm\infty}\dfrac{1}{t}\log\parallel\overline{\psi_t}(v) \parallel=\overline{\chi_i},~\forall~v\in\overline{E_i},~v\neq 0,~i=1,2,\cdots,k.$$ The numbers $\chi_1,\cdots,\lambda_k$ are \emph{\bf Lyapunov exponents} of $\overline{\psi_t}$ with respect to the ergodic measure $\mu$. According to the definition of the scaled linear Poincar\'e flow $\overline{\psi_t^*}:\mathcal{N}\rightarrow\mathcal{N}$, we have that $$\overline{\psi^*_t}(v)=\frac{\parallel X(x)\parallel}{\parallel X(\overline{\varphi_t}(x))\parallel}\overline{\psi_t}(v)=\frac{\overline{\psi_t}(v)}{\parallel\overline{\Phi_t}|_{\langle X(x)\rangle}\parallel},\quad\text{for each non-zero }v\in\mathcal{N}.$$ By the Poincar\'{e} Recurrence Theorem, for $\mu$-almost every $x\in M$, the Lyapunov exponent for $\overline{\Phi_t}$ along the flow direction is zero, namely,  $$\lim_{t\to\pm\infty}\dfrac{1}{t}\log\parallel\overline{\Phi_t}|_{\langle X(x)\rangle}\parallel=0.$$ Therefore,  $$\lim_{t\to\pm\infty}\dfrac{1}{t}\log\parallel\overline{\psi^*_t}(v)\parallel=\lim_{t\to\pm\infty}\dfrac{1}{t}\left(\log\parallel\overline{\psi_t}(v)\parallel-\log\parallel\overline{\Phi_t}|_{\langle X(x)\rangle}\parallel\right)=\lim_{t\to\pm\infty}\dfrac{1}{t}\log\parallel\overline{\psi_t}(v)\parallel.$$ Thus, for $\mu$-almost every $x$, we also have an integer $k\in[1,d-1]$, real numbers $\overline{\chi_1}<\overline{\chi_2}<\cdots<\overline{\chi_k}$ and a measurable $\overline{\psi^*_t}$-invariant splitting (for simplicity, we omit the base point) $$\mathcal{N}=\overline{E_1}\oplus\overline{E_2}\oplus\cdots\oplus\overline{E_k}$$ on normal bundle satisfying that $$\lim_{t\to\pm\infty}\dfrac{1}{t}\log\parallel\overline{\psi^*_t}(v) \parallel=\overline{\chi_i},~\forall~v\in\overline{E_i},~v\neq 0,~i=1,2,\cdots,k.$$ The numbers $\overline{\chi_1},\cdots,\overline{\chi_k}$ are \emph{\bf Lyapunov exponents} of $\overline{\psi^*_t}$ with respect to the ergodic measure $\mu$. It means that the Lyapunov exponents of the scaled linear Poincar\'{e} flow $\overline{\psi^*_t}$ and those of the linear Poincar\'{e} flow $\overline{\psi_t}$ are the same. Hence, the scaled linear Poincar\'{e} flow and the linear Poincar\'{e} flow also have the same Oseledec splitting.
   
   According to the relationship between $-X$ and $X$, one can get that $$\overline{\varphi_{-t}}=\varphi_t,\quad\text{ for every }t\in\mathbb{R}.$$ Therefore, for every nonzero vector $v\in T_xM$, we have that $$\lim_{t\to\pm\infty}\dfrac{1}{ t}\log\parallel\overline{\Phi_t}(v)\parallel=\lim_{t\to\pm\infty}\dfrac{1}{t}\log\parallel\Phi_{-t}(v)\parallel=-\lim_{t\to\pm\infty}\dfrac{1}{-t}\log\parallel\Phi_{-t}(v)\parallel.$$ Consequently, for every nonzero vector $v\in\mathcal{N}$, we also have that $$\lim_{t\to\pm\infty}\dfrac{1}{ t}\log\parallel\overline{\psi_t}(v)\parallel=\lim_{t\to\pm\infty}\dfrac{1}{t}\log\parallel\psi_{-t}(v)\parallel=-\lim_{t\to\pm\infty}\dfrac{1}{-t}\log\parallel\psi_{-t}(v)\parallel$$ and $$\lim_{t\to\pm\infty}\dfrac{1}{ t}\log\parallel\overline{\psi^*_t}(v)\parallel=\lim_{t\to\pm\infty}\dfrac{1}{t}\log\parallel\psi^*_{-t}(v)\parallel=-\lim_{t\to\pm\infty}\dfrac{1}{-t}\log\parallel\psi^*_{-t}(v)\parallel.$$ Thus, $$\overline{\chi}_i=-\chi_{k-i+1},\text{ for every }i=1,2,\cdots,k,$$ where $\chi_1,\cdots,\chi_{k}$ are the \emph{\bf Lyapunov exponents} of the linear Poincar\'{e} flow $\psi_t$ (scaled linear Poincar\'{e} flow $\psi^*_t$) with respect to $\mu$. Furthermore, we have that    $$\overline{E}_i=E_{k-i+1},~\overline{d}_i=\text{dim}\overline{E}_i=d_{k-i+1},~\text{ for }i=1,2,\cdots,k.$$ Regardless of the multiplicity of Lyapunov exponents, we have that $$\overline{\lambda_i}=-\lambda_{d-i},~\text{ for }i=1,2,\cdots,d-1.$$ Based on the above discussion, we can conclude that every ergodic hyperbolic invariant measure of the flow $\varphi_t$ is also an ergodic hyperbolic invariant measure of the flow $\overline{\varphi_t}$.

\begin{Proposition}\label{TSLE}{\rm (The approximation of smallest Lyapunov exponent)}
   Let $\mu$ be an ergodic hyperbolic invariant regular measure of the $C^1$ star vector field $X\in\mathfrak{X}^1(M)$ on a compact three dimensional smooth Riemannian manifold, $\lambda_1\leq\lambda_2$ be the Lyapunov exponents of $\mu$ with respect to the scaled linear Poincar\'{e} flow $\psi^*_t$. For every $\epsilon>0$, there is a hyperbolic periodic point $p$ of $\varphi_t$ such that $$\lvert\lambda_1-\lambda_1(p)\rvert<\epsilon,$$ where $\lambda_1(p)\leq\lambda_2(p)$ are the Lyapunov exponents of $\mu_p$ with respect to the scaled linear Poincar\'{e} flow $\psi^*_t$.     	
\end{Proposition}

\begin{proof}
   We consider the flow $\overline{\varphi_t}$ generated by the $C^1$ star vector field $-X\in\mathfrak{X}^1(M)$ on a compact three dimensional smooth Riemannian manifold. Then, $\mu$ is an ergodic hyperbolic invariant regular measure of the flow $\overline{\varphi_t}$. Since every star vector field admits a dominated splitting $\mathcal{N}_{\text{supp}(\mu)\backslash\text{Sing}(X)}=E\oplus F$ w.r.t the scaled linear Poincar\'{e} flow $\overline{\psi^*_t}$ such that $\text{\rm dim}(E)=\text{\rm Ind}(\mu)$, the Oseledec splitting $\mathcal{N}_\Gamma=\overline{E_1}\oplus\overline{E_2}$ of an ergodic hyperbolic invariant measure $\mu$ is a dominated splitting. According to the similar discussions on Lyapunov exponents as three different cases in the proof of Theorem \ref{ThmA}, we have that $$\lim_{t\to\pm\infty}\dfrac{1}{t}\log\parallel\overline{\psi^*_t}(v) \parallel=\overline{\lambda_j},~\forall~v\in\overline{E_j},~v\neq 0,~j=1,2,$$ where $\overline{\lambda_1}<0<\overline{\lambda_2}$. 
   
   Let $$\overline{\chi}=\min\{\lvert\overline{\lambda_1}\rvert,~\overline{\lambda_2}\}.$$ Since the Oseledec splitting $\mathcal{N}_\Gamma=\overline{E_1}\oplus \overline{E_2}$ is a dominated splitting, by Lemma \ref{time}, for every $0<\epsilon/4\ll\overline{\chi}$ and every $\delta_3\in(0,1/9)$, there is a positive number $\overline{L}=\overline{L}(\epsilon/4,\delta_3)$ such that
\begin{itemize}
   \item for every $T\geq\overline{L}$, there exists a measurable set $\overline{\Gamma^T}=\overline{\Gamma^T}(\epsilon/4,\delta_3)\subset\Gamma$ with $\mu\left(\overline{\Gamma^T}\right)\geq 1-\delta_3$;
 	
   \item there is a natural number $N=N(T)$ such that for every $n\geq N$ and every $x\in\overline{\Gamma^T}$, we have that $$\prod_{i=0}^{n-1}\left\lVert\overline{\psi^*_T}\lvert_{\overline{E_1}(\overline{\varphi_{iT}}(x))}\right\rVert\leq e^{-(\overline{\chi}-\epsilon/4)nT},~\prod_{i=0}^{n-1}m\left(\overline{\psi^*_T}|_{\overline{E_2}(\overline{\varphi_{iT}}(x))}\right)\geq e^{(\overline{\chi}-\epsilon/4)nT}.$$ $$\dfrac{\left\lVert\overline{\psi^*_T}|_{\overline{E_1}(x)}\right\rVert}{m\left(\overline{\psi^*_T}|_{\overline{E_2}(x)}\right)}\leq e^{-T(\overline{\chi}-\epsilon/4)}.$$     	  
 \end{itemize}
   Fix an integer $T_0\geq\overline{L}(\epsilon/4,\delta_3)$, let $\eta_0=(\overline{\chi}-\epsilon/4)T_0$, for $k\in\mathbb{N}^+$, we define the {\bf Pesin block } $\Lambda^{T_0}_{\eta_0}(k)$ with respect to the scaled linear Poincar\'{e} flow $\overline{\psi^*_t}$ as:
\begin{equation*}
\begin{aligned}
   \Lambda^{T_0}_{\eta_0}(k)=\Bigg\{x\in\Gamma:&\prod_{i=0}^{n-1}\left\lVert\overline{\psi^*_{T_0}}|_{\overline{E_1}(\overline{\varphi_{iT_0}}(x))}\right\rVert\leq ke^{-n\eta_0},~\forall~ n\geq 1,\\ &\prod_{i=0}^{n-1}m\left(\overline{\psi^*_{T_0}}|_{\overline{E_2}(\overline{\varphi_{iT_0}}(x))}\right)\geq ke^{n\eta_0},~\forall~ n\geq 1,~d(x,\text{Sing}(X))\geq\dfrac{1}{k}\Bigg\}.
\end{aligned}   
\end{equation*}
   According to \cite[Proposition 5.3]{WYZ}, $\Lambda^{T_0}_{\eta_0}(k)$ is a compact set and $$\mu\left(\Lambda^{T_0}_{\eta_0}(k)\right)\rightarrow\mu\left(\overline{\Gamma^{T_0}}\right)\text{ as } k\rightarrow+\infty.$$ Therefore, one can fix $k$ large enough, such that $\mu\left(\Lambda^{T_0}_{\eta_0}(k)\right)>1-2\delta_3>0$. For this fixed $k$, there is $j=j(k)\in\mathbb{N}^+$ such that $k<e^{\frac{jT_0\epsilon}{4}}$. Thus, for every $x\in\Lambda^{T_0}_{\eta_0}(k)$, one has that $$\prod_{i=0}^{n-1}\left\lVert\overline{\psi^*_{jT_0}}|_{\overline{E_1}(\overline{\varphi_{ijT_0}}(x))}\right\rVert\leq e^{-(\overline{\chi}-\epsilon/2)njT_0},~\prod_{i=0}^{n-1}m\left(\overline{\psi^*_{jT_0}}|_{\overline{E_2}(\overline{\varphi_{ijT_0}}(x))}\right)\geq e^{(\overline{\chi}-\epsilon/2)njT_0},\quad\text{for }\forall~ n\geq 1.$$ Set $\eta=(\overline{\chi}-\epsilon/2)jT_0$, $T=jT_0$, we consider the set
\begin{equation*}
\begin{aligned}
   \Lambda^{T}_{\eta}(k)=\Bigg\{x\in\Gamma:&\prod_{i=0}^{n-1}\left\lVert\overline{\psi^*_{T}}\lvert_{\overline{E_1}(\overline{\varphi_{iT}}(x))}\right\rVert\leq e^{-n\eta},~\forall~ n\geq 1,\\ &\prod_{i=0}^{n-1}m\left(\overline{\psi^*_{T}}|_{\overline{E_2}(\overline{\varphi_{iT}}(x))}\right)\geq e^{n\eta},~\forall~ n\geq 1,~d(x,\text{Sing}(X))\geq\dfrac{1}{k}\Bigg\}.
\end{aligned}   
\end{equation*}       

   Take $y\in\Lambda^{T}_{\eta}(k)$, by Poincar\'{e} Recurrence Theorem, there is an increasing integer sequence $\{l_n\}$ such that $$d(y,\overline{\varphi}_{l_nT}(y))\longrightarrow 0.$$ By a discussion similar to the {\bf Step 1, 2 } and {\bf 3} in the proof of Proposition \ref{TLLE}, we conclude that there exists a periodic point $p$ of $\overline{\varphi}_t$ such that $$\lvert\overline{\lambda_2}(p)-\overline{\lambda_2}\rvert<\epsilon,$$where $\overline{\lambda_1}(p)\leq\overline{\lambda_2}(p)$ are the Lyapunov exponents of $\mu_p$ with respect to the scaled linear Poincar\'{e} flow $\overline{\psi^*_t}$. According to the relationship between the Lyapunov exponents of the scaled linear Poincar\'{e} flow $\overline{\psi^*_t}$ and those of the scaled linear Poincar\'{e} flow $\psi^*_t$, we have that $$\overline{\lambda_1}(p)=-\lambda_2(p),\quad\overline{\lambda_2}(p)=-\lambda_1(p),\quad\overline{\lambda_1}=-\lambda_2,\quad\overline{\lambda_2}=-\lambda_1.$$ Thus, $$\lvert\lambda_1(p)-\lambda_1\rvert<\epsilon.$$  This completes the proof.    	         	
\end{proof} 

\begin{proof}[Proof of Theorem \ref{ThmB}]
   Given $\varepsilon>0$ small enough, let $$\epsilon=\varepsilon,\quad L=\max\left\{L\left(\frac{\epsilon}{4},\delta_1\right),\overline{L}\left(\frac{\epsilon}{4},\delta_3\right)\right\},$$ where $L\left(\frac{\epsilon}{4},\delta_1\right)$ given in Proposition \ref{TLLE}, $\overline{L}\left(\frac{\epsilon}{4},\delta_3\right)$ given in Proposition \ref{TSLE}. The set $$\Lambda:=\Gamma^T\cap\overline{\Gamma^T}$$ is a $\mu$-positive set. For the $\epsilon>0$, by the Propositions \ref{TLLE} and \ref{TSLE},  there is a periodic point $p$ such that the Lyapunov exponents of $\mu_p$ satisfy that $$\lvert\lambda_i-\lambda_i(p)\rvert<\varepsilon,~\text{ for every }i=1,2,$$ where $\lambda_1(p)\leq\lambda_2(p)$ are the Lyapunov exponents of $\mu_p$.      
\end{proof}


\noindent\text{Yuansheng Lu}\\
\noindent School of Mathematics and Statistics\\
\noindent Jiangsu Normal University, Xuzhou, 221116, P.R. China\\
\noindent L6yuansheng@163.com\\

\noindent\text{Wanlou Wu}\\
\noindent School of Mathematics and Statistics\\
\noindent Jiangsu Normal University, Xuzhou, 221116, P.R. China\\
\noindent wuwanlou@163.com\\

\end{document}